\documentclass[twoside, 11pt]{article}

\usepackage{graphicx, subfigure}
\usepackage[margin=1in, footnotesep=0.5in]{geometry}
\usepackage[sort&compress, numbers]{natbib} 
\usepackage{amsfonts, amsmath, amssymb, amsthm}
\usepackage{MnSymbol}
\usepackage{enumitem}
\usepackage{comment}
\usepackage{verbatim}

\usepackage{color}
\definecolor{darkred}{RGB}{100,0,0}
\definecolor{darkgreen}{RGB}{0,100,0}
\definecolor{darkblue}{RGB}{0,0,150}

\usepackage{hyperref}
\hypersetup{colorlinks=true, linkcolor=darkred, citecolor=darkgreen, urlcolor=darkblue}
\usepackage{url}

\def\hess{\cH}
\def\k{\mathsf{k}}
\def\L{\mathsf{L}}

\def\I{{\rm I}}
\def\d{{\rm d}}

\def\ball{B}

\newtheorem{thm}{Theorem}[section]

\newtheorem{proposition}[thm]{Proposition}

\newtheorem{lemma}[thm]{Lemma}

\newtheorem{corollary}[thm]{Corollary}
\newtheorem{contribution}{Contribution}

\theoremstyle{remark}

\newtheorem{remark}[thm]{Remark}

\numberwithin{equation}{section}
\numberwithin{figure}{section}

\def\beq{\begin{equation}} 
\def\eeq{\end{equation}}
\def\beqn{\begin{eqnarray*}}
\def\eeqn{\end{eqnarray*}}
\def\Bitem{\begin{itemize}\setlength{\itemsep}{.2in}}
\def\bitem{\begin{itemize}\setlength{\itemsep}{.05in}}
\def\eitem{\end{itemize}}
\def\Benum{\begin{enumerate}\setlength{\itemsep}{.2in}}
\def\benum{\begin{enumerate}\setlength{\itemsep}{.05in}}
\def\eenum{\end{enumerate}}
\def\bmult{\begin{multline*}}
\def\emult{\end{multline*}}
\def\bcenter{\begin{center}}
\def\ecenter{\end{center}}
\def\bframe{\begin{frame}}
\def\eframe{\end{frame}}

\newcommand{\prpref}[1]{Proposition~\ref{prp:#1}}

\newcommand{\lemref}[1]{Lemma~\ref{lem:#1}}
\newcommand{\secref}[1]{Section~\ref{sec:#1}}

\newcommand{\remref}[1]{Remark~\ref{rem:#1}}

\DeclareMathOperator{\dist}{dist}

\DeclareMathOperator{\trace}{trace}

\DeclareMathOperator{\diag}{diag}

\def\cE{\mathcal{E}}

\def\cH{\mathcal{H}}

\def\cN{\mathcal{N}}
\def\cO{\mathcal{O}}

\def\cS{\mathcal{S}}
\def\cT{\mathcal{T}}

\def\cV{\mathcal{V}}

\def\fB{\mathfrak{B}}
\def\fK{\mathfrak{K}}
\def\fS{\mathfrak{S}}

\def\bbI{\mathbb{I}}

\def\bbP{\mathbb{P}}
\def\bbQ{\mathbb{Q}}
\def\bbR{\mathbb{R}}

\newcommand{\E}{\operatorname{\mathbb{E}}}

\newcommand{\Cov}{\operatorname{Cov}}

\def\Unif{\text{Unif}}

\newcommand{\<}{\langle}
\renewcommand{\>}{\rangle}

\def\eps{\varepsilon}

\def\iff{\ \Leftrightarrow \ }

\def\1{\mathbbm{1}}
\newcommand{\IND}[1]{\bbI\{ #1 \}}

\definecolor{purple}{rgb}{0.4,.1,.9}

\begin{document}
\thispagestyle{empty}

\title{Embedding Functional Data:\\ Multidimensional Scaling and Manifold Learning\footnotetext{We would like to thank Annegret Burtscher, Thomas Diciccio, Wuchen Li, Joseph Romano, and Justin Roberts for helpful discussions. We are particularly grateful to Bruce Driver for proving \lemref{bruce} from scratch.
This work was partially supported by the US National Science Foundation (DMS 1821154, DMS 1916071).
}
}
\author{
Ery Arias-Castro\footnote{University of California, San Diego, California, USA (\url{https://math.ucsd.edu/\~eariasca/})} 
\and 
Wanli Qiao\footnote{George Mason University, Fairfax, Virginia, USA (\url{https://mason.gmu.edu/\~wqiao/})}
}
\date{}
\maketitle

\begin{abstract}
We adapt concepts, methodology, and theory originally developed in the areas of multidimensional scaling and dimensionality reduction for multivariate data to the functional setting. We focus on classical scaling and Isomap --- prototypical methods that have played important roles in these areas --- and showcase their use in the context of functional data analysis. In the process, we highlight the crucial role that the ambient metric plays. 

\end{abstract}

\section{Introduction} 
\label{sec:introduction}

{\em Functional data analysis (FDA)} is a specialized area in statistics that has developed around the need to analyze functional data, meaning, situations where observations are best modeled by functions as opposed to points. The associated literature is quite substantial, with several book-length expostions  
\cite{ferraty2006nonparametric, kokoszka2017introduction, ramsay2005functional, ramsay2002applied, hsing2015theoretical, ramsay2005fitting} and some review articles
\cite{wang2016functional, cuevas2014partial}.
FDA is closely related to longitudinal data analysis \cite{hall2006LDA, james2000principal, yao2005functional, zhou2008joint} --- which is more or less synonymous with situations in FDA where the data are sparse --- and overlaps with other areas such as time series clustering or classification \cite{abanda2019review, ismail2019deep, maharaj2019time, bagnall2017great, caiado2010classification, liao2005clustering},
image classification \cite{camps2013advances, lu2007survey, canty2014image}, shape analysis \cite{dryden2016statistical, srivastava2016functional, loncaric1998survey}, as well as 
signal alignment \cite{kneip1992statistical, wang1999synchronizing, trigano2011semiparametric} and image registration \cite{oliveira2014medical, hajnal2001medical, goshtasby2012image, perry2019sample}.

\subsection{Functional principal component analysis}
One of the earliest, and still one of the most popular, tools for the analysis functional data is {\em principal component analysis (PCA)}. 
It is ``the most prevalent tool in FDA" according to \citet{wang2016functional}, and features prominently in the classical textbooks by \citet{ramsay2005functional, ramsay2002applied}.
It corresponds to the Karhunen--Lo\`eve transformation, well-known in signal processing.

The two main uses of PCA in multivariate analysis are to produce an embedding into a lower-dimensional space and to construct new variables, the principal directions, which may lead to interesting relationships between the original variables. In the context of FDA, the principal directions are functions sometimes called {\em modes of variation}, that can be plotted (at least in the case of functions of one or two variables) for data exploration.

Our focus will be on the problem of embedding functional data, or in other words, {\em dimensionality reduction (DR)} in the context of FDA.
 
\subsection{Functional multidimensional scaling}
PCA is intimately related to, and in fact necessitates, the use of a Euclidean metric (in the multivariate setting) or Hilbertian metric (in the functional setting). 
In the FDA statistics literature, in particular, the $L_2$ metric appears to be the most prevalent.
Other metrics have, however, been considered in the context of functional data, even non-Hilbertian ones, e.g., for measuring the similarity between images \cite{daly1992visible, van1996perceptual, winkler1998perceptual, wang2003multiscale}.
Also, in statistics theory, many metrics and divergences have been suggested for comparing functions or densities --- which are often used to model functional data. 

When using a metric or, more generally, a dissimilarity, that is not Hilbertian, an embedding may be obtained via a method for {\em multidimensional scaling (MDS)}.
Strictly speaking, MDS is the problem of embedding `objects' based on proximity information, and is thus distinct from DR --- which is really the problem we are interested in. Having said this, any method for MDS can be turned into a method for DR, by computing all pairwise dissimilarities and then applying the method to these dissimilarities.

While PCA (attributed to \citet{pearson1901liii} and \citet{hotelling1933analysis2,hotelling1933analysis1}) is the main method for DR, {\em classical scaling (CS)} (attributed to \citet{torgerson1958theory} and \citet{gower1966some}) is the main method for MDS. In fact, even though they address different problems, the two methods are sometimes confused, as they yield the same embedding in a Euclidean setting: PCA applied to points in a Euclidean space to obtain an embedding in dimension $k$, say, is equivalent to CS applied to the corresponding pairwise Euclidean distances to produce an embedding in the same dimension $k$. Even though this is true, CS may also be used with dissimilarities that are not Euclidean or Hilbertian.

\begin{contribution}
We examine the behavior of classical scaling with different choices of metric or divergence such as the Wasserstein metric and the Kullback--Leibler divergence. We do so in the context of some emblematic models such as location--scale families and exponential families of densities.
\end{contribution}

\subsection{Functional manifold learning}
However important and popular PCA remains, there are other methods for DR that have been developed for multivariate data. The modern literature is found under umbrella names such as {\em nonlinear dimensionality reduction} or {\em manifold learning}, which describe the basic working situation in which the data points are on or near a submanifold admitting a global chart. The goal, then, is to recover such a chart, or at least the embedding of the data points that it provides.


The {\em isometric feature mapping (Isomap)} method of \citet{Tenenbaum00ISOmap} is particularly emblematic. Its foundation is the added assumption that the underlying manifold is isometric to a convex domain in some Euclidean space, or put differently, that it admits a global chart that provides an isometry between the manifold and a convex set. Founded on this assumption, Isomap proceeds by estimating the pairwise {\em intrinsic distances} and then applying CS to these estimated distances. (We describe the algorithm in more detail in \secref{isomap}.)

In the context of FDA, manifolds and their intrinsic distances (and other geometrical attributes) have been studied in statistics and information theory in the context of {\em information geometry} \cite{amari2016information, ay2017information}, which is a literature that studies statistical models from a geometrical perspective.
In that literature, it is well-known that a broad class of divergences that includes the Hellinger distance and the (symmetrized) Kullback--Leibler divergence induce on smooth statistical models the same intrinsic metric: the {\em Fisher metric} (aka {\em Fisher--Rao metric}) introduced by \citet{radhakrishna1945information, rao1987differential}.

The Fisher metric has found multiple uses in applications calling for functional modeling, such as the detection of structure in images \cite{maybank2006application, maybank2004detection, maybank2019fisher}, in shape analysis \cite{peter2006shape}, psychometrics \cite{da2016derivation}, in proteomics \cite{tucker2014analysis}, and in neuroscience \cite{wu2014analysis}, among other fields.
\citet{carter2009fine} make a connection between the Fisher metric and Isomap, and showcase the use of the latter for visualization and classification of document data and biological (clinical flow cytometry) data.
In the FDA statistics literature proper, \citet{chen2012nonlinear} propose Isomap as a method for DR and provide some elements of theory.

\begin{contribution}
We examine the behavior of Isomap under different metrics and divergences, some of them inducing the Fisher metric on certain smooth models such as exponential families of densities.
\end{contribution}

\subsection{Content} 

The remainder of the paper is organized as follows.
In \secref{general}, we describe the setting that we consider, which consists of a sample of densities from an underlying, unknown statistical model. These densities may be entirely available or only partially known by way of samples --- the latter being a common distinction in FDA. The setting encompasses a number of important settings that we detail later in the paper. In this broad context, we establish some consistency for CS and for Isomap. 
In \secref{examples}, we specialize these general results to various metrics and divergences applied to various parametric models of densities --- in particular, location--scale families and exponential families.
More specifically, in \secref{L2}, we consider some Hilbertian metrics, in particular, the $L_2$ metric and RKHS metrics.
In \secref{hellinger}, we consider a class of regular divergences that includes the Hellinger distance and the Kullback--Leibler divergence, which on smooth models induce the Fisher metric in their intrinsic form. 
In \secref{W2}, we consider the Wasserstein $W_2$ metric.
\secref{discussion} is a discussion section.
All the technical arguments are gathered in \secref{proofs}.

\section{General results}
\label{sec:general}

In this section we describe the setting that we consider in most, if not all of our examples, and state some general results on the consistency of CS and Isomap that will be applied to particular cases in the next section. 

\subsection{Setting}
\label{sec:setting}

The basic setting is that of a set of densities that we need to embed as points in some pre-specified Euclidean space. 
We will consider the following two situations:

\begin{itemize}
\item {\em Population setting}\quad 
Here the data consist in $n$ densities, $q_1, \dots, q_n$, with respect to some known measure $\lambda$ on $\bbR^d$. The densities are completely available, meaning that, in principle, we can compute the value of any functional applied to these densities. 
\item {\em Sample setting}\quad 
Here the data consist in $n$ samples, $\cS_1, \dots, \cS_n$, of respective sizes $m_1, \dots, m_n$, from $n$ underlying densities respect to $\lambda$. Thus, in this setting, the densities are only available via these samples, and the value of  a functional applied to these densities can only be estimated. 
\end{itemize}

The embedding methodology that we consider is not based on any specific modeling, but as is often the case, it will be evaluated on particular models of importance in the statistical literature. Concretely, we will examine situations where the densities belong to a model of the form $\{f_\theta : \theta \in \Theta\}$, where $f_\theta$ is a density with respect to $\lambda$ --- either the Lebesgue measure or the counting measure in our examples. When finite-dimensional --- the `parametric' situation --- the parameter space $\Theta$ is a subset of $\bbR^p$, but in principle $\Theta$ may be infinite-dimensional. We may use the notation $f_\theta(\cdot) \equiv f(\cdot, \theta)$ on occasion. This setting is quite general and covers the vast majority of the practical situations. We will provide some specific examples in \secref{models}. 

\subsection{Embedding problem}
\label{sec:embedding problem}
Consider a general situation in which we have $q_1, \dots, q_n \in \bbQ$, where $\bbQ$ is some set that is equipped with a dissimilarity $\delta: \bbQ \times \bbQ \to \bbR_+$. Until we discuss some generalizations in \secref{discussion}, we ask a dissimilarity to satisfy $\delta(q, q_0) = 0 \iff q = q_0$ and $\delta(q,q_0) = \delta(q_0,q)$.

The embedding problem in this context can be stated in general terms as follows: 
\begin{quote}
Given an embedding dimension $d_e$ (often $d_e=2$ when the goal is to visualize the data), find $p_1, \dots, p_n \in \bbR^{d_e}$ such that $\|p_i - p_j\| \approx \delta_{ij} := \delta(q_i, q_j)$ for all (or most) $i, j \in [n]$. 
\end{quote}
The quality of the approximation can be formalized in multiple ways, for example, via the following notion of {\em stress}
\begin{equation}
\label{stress}
\sum_{i, j} \big|\|p_i - p_j\|^2 - \delta_{ij}^2\big|.
\end{equation}
This is, in fact, a variant of the raw stress function commonly used in MDS. We chose this variant because of its intimate relationship with PCA.

\subsection{Dissimilarities based on an inner product}
\label{sec:inner product}
The use of the $L_2$ norm is common in FDA as it allows for the application of PCA. This is the dissimilarity that we consider now, being defined as 
\begin{align}
\delta(q, q_0)^2 = \|q - q_0\|^2 = \textstyle\int (q-q_0)^2 \d\lambda.
\end{align}
The corresponding inner product is denoted 
\begin{align}
\<q, q_0\> = \textstyle\int q q_0 \d\lambda.
\end{align}
When needed, we assume that the densities under consideration are square integrable.

We start by studying PCA, porting well-known results in the multivariate setting to the functional setting. And we then draw conclusions for CS, which is made possible by the fact that the two methods return the same embedding in the present setting.
 
\subsubsection{Principal component analysis}
\label{sec:PCA}

In the population setting of \secref{setting}, PCA computes the eigenfunctions of the integral operator with kernel
\begin{align}
\label{kappa}
\kappa(s,t) := \sum_{i=1}^n q_i(s) q_i(t) - n \bar q(s) \bar q(t), &&
\text{where} \quad \bar q(s) := \frac1n \sum_{i=1}^n q_i(s),
\end{align}
and obtains an embedding by projecting each $q_i$ onto the subspace spanned by a set of eigenfunctions for this operator for its top $d_e$ eigenvalues. 
With the embedding dimension $d_e$ left implicit in the background, PCA satisfies the following optimality property.

\begin{proposition}
\label{prp:PCA stress}
In the population setting of \secref{setting}, PCA returns an orthogonal projection with minimum stress (among orthogonal projections of same rank). 
\end{proposition}

We note that, if the $d_e$ and $d_e+1$ largest eigenvalues of the operator defined by $\kappa$ coincide, the choice of projection is not unique, but the resulting embeddings are all rigid transformations of each other. 
This is due to \prpref{PCA stress} and the fact that all optimal (orthogonal) projections result in embeddings that are rigid transformations of each other.
In that sense we may say that there is a unique orthogonal projection that minimizes the stress among orthogonal projections, and it is recovered by PCA. 
Although we could not find this result as stated despite a number of early publications exploring basic properties of PCA in a functional setting, e.g., \cite{castro1986principal, besse1986principal}, it is a straightforward consequence of a correspondence with the multivariate data analysis setting. We provide a succinct proof in \secref{proofs}.

While we may have to work with samples instead of densities, the following consistency result holds. Recall that the embedding dimension $d_e$ is fixed and left implicit in the background.

\begin{proposition}
\label{prp:PCA consistency}
In the sample setting of \secref{setting}, suppose that based on the samples we produce $\hat q_1, \dots, \hat q_n$ that are consistent for $q_1, \dots, q_n$ in $L_2(\lambda)$. Then PCA applied to the resulting $\hat\kappa$ is consistent in the sense that the orthogonal projection it returns, $\hat p_1,\dots,\hat p_n$, is asymptotically stress minimizing. (The asymptotic limit is as $\min_j m_j \to \infty$ while $n$ remains fixed.)
In fact, there is $C>0$ depending only on the configuration $q_1, \dots, q_n$ such that
\begin{align}
\label{PCA bound}
\min_{\{p_1,\dots,p_n\}} \sum_{i=1}^n\|p_i - \hat p_i \|^2 \le C  n\sum_{i=1}^n \| q_i - \hat q_i\|^2  \Big( 1 + \sum_{i=1}^n \| q_i - \hat q_i\|^2 \Big),
\end{align} 
where the minimum is over the set of stress minimizing orthogonal projections.
\end{proposition}

Again, the result is essentially known, at least in the multivariate setting, but as it is not readily available in the functional setting we provide a succinct proof in \secref{proofs}.

\prpref{PCA consistency} establishes the consistency of PCA in the asymptotic limit where $\min_j m_j \to \infty$ while $n$ remains fixed. Such a result is particularly meaningful when the $n$ underlying populations, $q_1, \dots, q_n$, are, themselves, of interest.
A situation which seems more common in FDA is where these densities are not of particular interest because they are drawn from a larger population of densities. This is invariably the case in the context of longitudinal data, for example. In such a situation, it is of possibly greater interest to consider what happens when $n\to\infty$. In order to consider this situation, we assume that the densities are iid copies of a stochastic process. 

\begin{proposition}
\label{prp:PCA n-consistency}
In the population setting of \secref{setting}, suppose that $Q_1, \dots, Q_n$ are iid copies of a stochastic process $Q$ with values in $L_2(\lambda)$, and $\hat \pi$ is the PCA orthogonal projection based on the $Q_i$. Let $\Pi$ be the set of all
orthogonal projections $\pi$ onto a $d_e$-dimensional subspace of $L_2(\lambda)$ minimizing the following notion of expected stress over such projections
\begin{align}
\label{expected_stress}
\E\Big[\big|\|\pi(Q) -\pi(Q')\|^2 - \|Q-Q'\|^2\big|\Big],
\end{align}
where $Q$ and $Q'$ are iid copies. 
Then, for any $q \in \bbQ$, with probability one in the asymptotic limit where $n\to\infty$,   
\begin{align}
\min_{\pi\in\Pi}\|\pi(q) - \hat\pi(q)\|^2 \to 0.
\end{align}
If $\E \|Q\|^4 < \infty$, there exists a constant $C>0$ depending on $Q$ such that for any $q \in \bbQ$ and $n\geq1$,
\begin{align}
\E\min_{\pi\in\Pi} \|\pi(q) - \hat\pi(q)\|^2 < C \|q\|^2 n^{-1}.
\end{align}
\end{proposition}

The asymptotic limit where $n\to\infty$ may also be studied in the sample setting of \secref{setting}. We note that, perhaps surprisingly at first sight, a consistency result may be derived even when $m_1, \dots, m_n$ remain bounded, as done, e.g., in \cite{hall2006LDA}. Such an asymptotic setting may be seen as a most extreme form of sparse FDA, perhaps encountered in longitudinal data analysis. We do not provide additional details.

\subsubsection{Classical scaling}
\label{sec:CS inner}
As is well-known, when applied in the context of a Hilbert space, CS is equivalent to PCA, in that the two methods produce the same embedding (again, up to a rigid transformation). Perhaps for this reason, the two methods are sometimes confused. Although leading to the same embedding, they take different computational paths to get there. Indeed, continuing with the same notation, CS proceeds by computing the top $d_e$ eigenvectors of the matrix $B := (\<q_i, q_j\>)$, obtaining $u_1, \dots, u_{d_e}$ ordered according to the eigenvalues $\nu_1 \ge \dots 
\ge \nu_{d_e}$, and embeds $q_i$ as $(\sqrt{\nu_1} u_{i,1}, \dots, \sqrt{\nu_{d_e}} u_{i,d_e}) \in \bbR^{d_e}$, where $u_j = (u_{1,j}, \dots, u_{n,j})$.

Even though the setting may be infinite dimensional in principle, the fact that we only have finitely many densities renders the problem effectively multivariate. And, in the multivariate setting, it is an established fact that the two methods return the same output. (To be sure, we provide some technical details in \secref{proofs}.)
Knowing this, we may draw the following conclusions from the results in \secref{PCA}.

\begin{corollary}
\label{cor:CS inner}
In the population setting of \secref{setting}, CS returns an embedding that corresponds to an orthogonal projection minimizing the stress.
\end{corollary}

\begin{corollary}
\label{cor:CS inner consistency}
In the context of \prpref{PCA consistency}, CS applied to $\hat B := (\<\hat q_i, \hat q_j\>)$ is consistent in the sense that it asymptotically recovers a stress minimizing orthogonal projection.
\end{corollary}

\begin{corollary}
\label{cor:CS inner n-consistency}
In the context of \prpref{PCA n-consistency}, the same conclusions apply to CS.
\end{corollary}

While this is stated in the context of the $L_2$ norm, we mention very recent work of \citet{lim2022classical}, which establishes consistency under more general conditions (see Corollary 7.4 there). 

\subsection{General dissimilarities}
\label{sec:general dissimilarities}
We still consider the same generic embedding problem described in \secref{embedding problem}, except that now the dissimilarity $\delta$ is general --- although in our examples it will either be a metric or a well-behaved divergence.

\subsubsection{Classical scaling}
\label{sec:classical scaling}
When $\delta$ is not based on an inner product, the motivation for using PCA is not clear as the procedure does not take the dissimilarity $\delta$ into account to produce an embedding. 
However, CS remains relevant. 
For a general dissimilarity $\delta$, it takes the following form: 
\begin{enumerate}[noitemsep]
\item Form the matrix $A = (a_{ij})$ with $a_{ij} := -\frac12 \delta_{ij}^2$; 
\item Double-center $A$ to obtain $B = (b_{ij})$ with $b_{ij} := a_{ij} - \bar a_{i\cdot} - \bar a_{\cdot j} + \bar a_{\cdot\cdot}$;
\item Compute the top $d_e$ eigenvectors of $B$, denoted $u_1, \dots, u_{d_e}$ and ordered according to the eigenvalues $\nu_1 \ge \dots \ge \nu_{d_e}$; 
\item Embed $q_i$ as $(\sqrt{\nu_1^+}u_{i,1}, \dots, \sqrt{\nu_{d_e}^+} u_{i,d_e}) \in \bbR^{d_e}$, where $u_j = (u_{1,j}, \dots, u_{n,j})$.
\end{enumerate}
Above, $a^+ = \max(a, 0)$ for any $a \in \bbR$, which is the positive part of $a$.
This is necessary when $\delta$ is a general metric, as unlike when it is based on an inner product, the matrix $B$ is not necessarily positive semidefinite. In fact, by a classical theorem of \citet{schoenberg1935remarks}, $B$ is positive semidefinite exactly when there is a configuration of points in a Euclidean space (of dimension anywhere between $1$ and $n$) whose pairwise distances coincide with the $\delta_{ij}$.

In general, CS is not known to satisfy an optimality property in terms of stress, but instead, in terms of the {\em strain}, defined as
\begin{equation}
\label{strain}
\sum_{i, j} \big(\<p_i, p_j\> - b_{ij}\big)^2,
\end{equation}
where $(b_{ij})$ is defined above.
Indeed, CS returns an embedding in the desired dimension that minimizes the strain. However, the strain is not nearly as intuitive as the stress.
That being said, CS can nonetheless be used to initialize an iterative algorithm that aims at minimizing the stress incrementally, for example, the SMACOF algorithm of \citet{de2009multidimensional}. 

We now consider the question of consistency. We do this in the following two propositions. 
We start with the fixed-$n$ asymptotic regime. 

\begin{proposition}
\label{prp:CS consistency}
In the sample setting of \secref{setting}, suppose that based on the samples we produce $(\hat \delta_{ij})$ consistent for $(\delta_{ij})$. Let $(\hat\nu_k,\hat u_{k})$, $k=1,\dots,d_e$, be the top $d_e$ eigenpairs of $\hat B$, and $\hat p_i:=(\sqrt{\hat \nu_1^+}\hat u_{i,1}, \dots, \sqrt{\hat\nu_{d_e}^+} \hat u_{i,d_e})$. Then CS applied to $(\hat \delta_{ij})$ is consistent in the sense that it asymptotically ($m_j \to \infty$ for all $j$, with $n$ fixed) recovers an embedding that results from applying CS to $(\delta_{ij})$, for which the rate of convergence is determined as follows: For some $C>0$ only depending on $(\delta_{ij})$ and $d_e$, 
\begin{align}
\label{hatu consistency}
\min_{\{p_1,\dots,p_n\}} \sum_{i=1}^n\|p_i - \hat p_i\|_2^2 \le Cn^2 
\begin{cases}
\|A-\hat A\|_2^2, & \text{if } \nu_{d_e} > 0; \\
\|A-\hat A\|_2, & \text{if } \nu_{d_e} = 0.
\end{cases}
\end{align}
where the minimum is over the set of embedding points minimizing the strain in \eqref{strain}.
\end{proposition}

We now consider the $n\to\infty$ asymptotic regime in the form of the following proposition. Some related consistency (and very recent) results are available in \cite[Sec 5]{kroshnin2022infinite} and in \cite[Sec 7]{lim2022classical}. (In fact, in our proof arguments rely, in part, on the former.)
 
\begin{proposition}
\label{prp:CS n-consistency}
In the population setting of \secref{setting}, and with $\delta$ being a metric, let $\bbQ$ be a subset of densities with respect to $\lambda$ that is compact and separable for $\delta$. In this context, suppose that $Q_1, \dots, Q_n$ are iid copies of a stochastic process $Q$ supported on $\bbQ$ such that
\begin{align}
\label{measure zero}
\bbP\big(\delta^4(q_1, Q) - \delta^4(q_2, Q) = s\big) = 0, \quad \text{for all $q_1 \ne q_2$ in $\bbQ$ and all $s \in \bbR$}.
\end{align}
Then, there exists a function $\pi_n: \bbQ \to \bbR^{d_e}$ coinciding with CS when applied to $Q_1, \dots, Q_n$ such that, with probability one, 
\begin{align}
\E_Q \Big[\min_{\pi\in\Pi} \|\pi_n(Q) - \pi(Q)\|^2\Big] \to 0, \quad \text{as } n\to\infty,
\end{align}
where $\Pi$ is the collection of functions $\pi: \bbQ \to \bbR^{d_e}$ which minimize the following notion of expected strain
\begin{align}
\E\Big[\big(\<\pi(Q), \pi(Q')\> - b(Q,Q')\big)^2\Big],
\end{align}
where 
\begin{align}
\label{bqdef}
b(q,q') = -\frac12 \Big(\delta(q,q')^2 - \E[\delta(q,Q')^2] - \E[\delta(Q,q')^2] + \E[\delta(Q,Q')^2]\Big),
\end{align}
$Q$ and $Q'$ being iid copies.
\end{proposition}

We anticipate the condition \eqref{measure zero} --- which asks the distribution of $Q$ to put zero mass on what corresponds to algebraic surfaces in the Euclidean setting --- to be mild. It is, for example, satisfied in the context of a location--scale or an exponential family of densities (see \secref{models}) when the stochastic process results from sampling the space parameterizing the family with a Lebesgue density having compact support. This is true for all the metrics studied in \secref{examples}.


\subsubsection{Isomap}
\label{sec:isomap}
Isomap was proposed by Tenenbaum et al \cite{Tenenbaum00ISOmap, silva2002global} for the problem of embedding points in a Euclidean space thought to be on or close to a smooth surface --- the manifold learning problem. It can also be applied in the more general setting of \secref{embedding problem}, where it takes the following form:
\begin{enumerate}[noitemsep]
\item Form a graph with node set $\{q_1, \dots, q_n\}$ and edge set $\{(q_i,q_j): \delta_{ij} \le r\}$, and weigh the edge $(q_i,q_j)$ by $\delta_{ij}$;
\item Compute $(d_{ij})$, where $d_{ij}$ is the shortest-path distance in the graph between nodes $i$ and $j$;
\item Apply CS to $(d_{ij})$.
\end{enumerate}
The connectivity radius $r$ is a tuning parameter of the method.

The original motivation for the first two steps is to estimate the intrinsic distances on the surface. \citet{bernstein2000graph} established some theoretical foundation for this early on; see also \cite{arias2019unconstrained, arias2020minimax, arias2020perturbation} and references therein.
Once the intrinsic distances are computed, a call to CS is made to produce an embedding.
Working with the intrinsic dissimilarity instead the dissimilarity itself is compelling, in particular in a functional setting where the ambient space --- for example, all the densities in $L_2(\lambda)$ as in \secref{PCA} and \secref{CS inner} --- is gigantic.

\begin{remark}
With hindsight, Isomap can be seen as applying the MDS-D method of \citet{kruskal1980designing} for embedding a graph to the neighborhood graph constructed in the first step. (MDS-D reappeared later in the form of the MDS-MAP method of \citet{shang2003localization}.) 
\end{remark}

Let $\bbQ$ be some set that equipped with a metric $\delta: \bbQ \times \bbQ \to \bbR_+$. In that space, the length of a path (i.e., a continuous curve) $\gamma: [a, b] \to \bbQ$ is given by
\begin{align}
\label{length}
\L(\gamma) := \sup_{a = t_0 < \cdots < t_k = b}\ \Sigma(\gamma, t_1, \dots, t_k), && \Sigma(\gamma, t_1, \dots, t_k) := \sum_{i=1}^k \delta(\gamma(t_{i-1}), \gamma(t_i)).
\end{align}
Note that, because we are requiring $\delta$ to be a metric, if we refine $t_1, \dots, t_k$ by inserting $t$ between $t_{i-1}$ and $t_i$, we can only increase $\Sigma(\gamma, \cdot)$, as $\delta(\gamma(t_{i-1}), \gamma(t_i)) \le \delta(\gamma(t_{i-1}), \gamma(t)) + \delta(\gamma(t), \gamma(t_i))$. Therefore, the supremum may be taken over sequences with maximum spacing bounded by any arbitrary $\eta > 0$.
We say that $\gamma$ connects $q,q' \in \bbQ$ if $\gamma(a) = q$ and $\gamma(b) = q'$. 
The intrinsic metric induced by $\delta$ is then defined as
\begin{align}
\delta_\L(q,q') = \inf\big\{\L(\gamma) : \gamma \text{ connects } q, q'\big\}.
\end{align}

Even though $\delta_\L$ is a true metric, in general, some strange things can happen: for example, it is possible that $\delta_\L(q,q_0) = \infty$ for all $q \ne q_0$ in $\bbQ$ (in fact, we will encounter this situation later on); $\delta_\L$ may induce a topology which is very different from the one induced by $\gamma$; and --- although this is much less important for us here --- the existence of shortest paths is not guaranteed in general.
For more on these notions, including examples exhibiting one or more of these issues, see \cite[Ch~2]{burago2001course}.
In all our examples, the situation will be tame for the most part. Even then, the basic consistency results below do not rely on that.

Suppose we have available $q_1, \dots, q_n \in \bbQ$. We proceed as in Isomap to estimate the intrinsic metric using graph distances. Having constructed a neighborhood graph on these points with connectivity radius $r$, we estimate $\delta_\L(q_i,q_j)$ by the shortest-path distance in that neighborhood graph between $q_i$ and $q_j$, denoted $d_{ij}$ above.

The first consistency result says that, if the sample $q_1, \dots, q_n$ becomes dense in $\bbQ$, and the connectivity radius is made to tend to zero slowly enough, then the graph distances are consistent for the intrinsic distances.
We note that obtaining rates is possible, but when more structure is in place; see \cite{bernstein2000graph, arias2019unconstrained}. A metric space is called proper if every closed ball included in it is compact. We assume that $(\bbQ,\delta)$ is proper for the following two propositions. 

\begin{proposition}
\label{prp:Isomap n-consistency}
In the population setting of \secref{setting}, suppose that $q_1, \dots, q_n$ are such that 
\begin{align}
\label{dense}
\eps_n := \sup_{q\in \bbQ} \min_{i \in [n]} \delta(q, q_i) \to 0.
\end{align}
Assume that 
\begin{align}
\label{delta approx}
\text{$\delta_\L(q,q_0) = \delta(q,q_0)[1 + \omega(q,q_0)]$, where $\omega$ is continuous with $\omega(q,q) = 0$ for all $q$.} 
\end{align}
Fix $i$ and $j$ such that $\delta_\L(q_i, q_j) < \infty$. Then, with a choice of connectivity radius $r_n \to 0$ slowly enough that $\eps_n/r_n \to 0$, $d_{ij}$ is consistent for $\delta_\L(q_i, q_j)$ as $n \to \infty$.
\end{proposition}

The second consistency result is similar, but deals with the sample setting instead of the population setting.

\begin{proposition}
\label{prp:Isomap sample n-consistency}
In the sample setting of \secref{setting}, assume that \eqref{dense} and \eqref{delta approx} hold, and in addition, that based on the samples we produce $(\hat \delta_{ij})$ consistent for $(\delta_{ij})$, uniformly over $i$ and $j$, in the following sense
\begin{align}
\label{delta uniformly consistent}
\text{$\max_{i \ne j} |\hat\delta_{ij} - \delta_{ij}|/\delta_{ij} \to 0$ in probability as $\min_i |\cS_i| \to \infty$.}
\end{align}
Fix $i$ and $j$ such that $\delta_\L(q_i, q_j) < \infty$. Then, with a choice of connectivity radius $r_n \to 0$ slowly enough that $\eps_n/r_n \to 0$, $\hat d_{ij}$ is consistent for $\delta_\L(q_i, q_j)$ as $n \to \infty$.
\end{proposition}

The following result will be useful when deriving an induced intrinsic metric. If the result does apply and the tensor $A$ in \eqref{intrinsic limit} is constant, then the induced intrinsic metric is Euclidean. 
Below, ${\rm GL}(\bbR^p)$ denotes the general linear group of $\bbR^d$.

\begin{proposition}
\label{prp:intrinsic limit}
Let $\bbQ$ denote an open connected subset of $\bbR^p$ equipped with its Euclidean norm $\|\cdot\|$. 
Assume that
\begin{gather}
\text{$\delta$ is equivalent to the Euclidean metric,} \label{delta equivalent}
\end{gather}
and that
\begin{gather}
\begin{gathered}
\delta(q,q_0)^2 = (q-q_0)^\top A(q_0) (q-q_0) [1 + \omega(q,q_0)], \label{intrinsic limit} \\
\text{where $A: \bbQ \to {\rm GL}(\bbR^p)$ is continuous,} \\
\text{and $\omega:\bbQ\times\bbQ \to \bbR$ is continuous with $\omega(q,q) = 0$.}
\end{gathered}
\end{gather}
Then $\delta_\L$ coincides with the Riemannian metric on $\bbQ$ defined by the tensor $A$. 
And \eqref{delta approx} holds.
\end{proposition}

\begin{remark}
\label{rem:intrinsic limit}
Property \eqref{delta equivalent} holds, for example, when $\Theta$ is bounded and $\delta$ can be extended to a metric on $\bar\Theta$ and, as a function, is continuous.
Property \eqref{intrinsic limit} holds, for example, when $\delta$ is twice continuously differentiable with nonsingular Hessian, in which case $A(q_0) = \frac12 H(q_0,q_0)$, where $H$ is the $p\times p$ top diagonal block of the Hessian matrix.
\end{remark}


\section{Examples}
\label{sec:examples}

In this section, we go through some emblematic examples of statistical models and consider the use of various classical metrics and divergences for the purpose of embedding in the context of \secref{setting}. Effectively, we examine the metric and intrinsic metric induced on the model. 

As we are considering two particular embedding methods, CS and Isomap, a central question is whether they produce an embedding that is accurate. We note that this is a tall order for any method because the embedding is constrained to be in a Euclidean space of given dimension.

\begin{itemize}
\item {\bf Classical scaling}\quad
In \secref{CS inner} and \secref{classical scaling}, we studied the behavior of CS under the $L_2$ and other general metrics, both in the population and sample settings. In particular, CS is exact (population setting) and consistent (sample setting) when the metric induced on the parameter space is Euclidean. 
\item {\bf Isomap}\quad
In \secref{isomap}, we studied the behavior of Isomap under a general (ambient) metric, both in the population and sample settings. In \prpref{Isomap n-consistency} (population setting) and in \prpref{Isomap sample n-consistency} (sample setting), we saw conditions under which the estimation of intrinsic distances based on graph distances is consistent. Assuming this is the case, still, Isomap can only be consistent if the intrinsic metric induced on the parameter space is Euclidean. This is because Isomap applies CS to the estimated intrinsic distances. 
\end{itemize}
Since the parameterization is arbitrary, we will be particularly interested in the question of whether there is a parameterization of the underlying model that makes the induced (resp.~intrinsic) metric Euclidean.
In probing this question, we will typically start with an arbitrary parameterization of the model (although a smooth one if possible) and then consider the question from the perspective of finding a re-parameterization that has the desired property.

\subsection{Examples of models}
\label{sec:models}
The general setting of a dominated statistical model covers a lot of important situations. This is a setting where we have 
\begin{align}
\label{model}
\text{a family of densities $\big\{f_\theta : \theta \in \Theta\big\}$ on $\bbR^d$},
\end{align}
with $\Theta$ a subset of a Euclidean space $\bbR^p$ --- assumed to be open and connected unless otherwise stated. 
The densities are with respect to a measure $\lambda$ on $\bbR^d$, which in our examples is either the Lebesgue measure or the counting measure.
The model will be assumed identifiable unless otherwise stated.
This is the framework that will consider in most of our examples, but not all, as we also consider an infinite-dimensional model.

We pause to note that whether the underlying model is parametric or not is not directly irrelevant, at least from a methodology perspective, as the methods that we consider are agnostic to the assumed model. 
The models are simply used as benchmarks to evaluate the behavior or performance of the methods. 

\subsubsection{Location--scale families}
A location–scale family is of the form
\begin{align}
\label{loc scale}
f_\theta(x) = |\theta_{\rm s}|^{-1} f(\theta_{\rm s}^{-1} (x - \theta_{\rm l})),
\end{align}
with $\theta = (\theta_{\rm l}, \theta_{\rm s}) \in \Theta_{\rm l} \times \Theta_{\rm s}$, where $\Theta_{\rm l}$ is typically a subgroup of $(\bbR^d, +)$ and $\Theta_{\rm s}$ is typically a subgroup of ${\rm GL}(\bbR^d)$, while $f$ is some density with respect to the Lebesgue measure. 
(For a square matrix $A$, $|A|$ denotes its determinant.)
In standard treatments, the density $f$ is known, but again, this is irrelevant here as the methodology is blind to this. 
The quintessential example of a location--scale family is the normal family of distributions. 

The family is a location family if $\Theta_{\rm s}$ is trivial, in which case it takes the form
\begin{align}
\label{loc}
f_\theta(x) = f(x - \theta),
\end{align}
with $\theta \in \Theta = \bbR^d$ by default; and it is a scale family if, instead, $\Theta_{\rm l}$ is trivial, in which case it takes the form
\begin{align}
\label{scale}
f_\theta(x) = |\theta|^{-1} f(\theta^{-1} x),
\end{align}
with $\theta \in \Theta = {\rm GL}(\bbR^d)$ by default.

A location–scale family is arguably the most basic type of distribution family, but it plays an important role as, in applications, location parameters such as the mean or median (and other quantiles) are often of interest. In FDA and the closely related literature on signal or image registration, such families provide simple but fairly rich models for time-warping \cite[Ch~7]{ramsay2005functional}. The problem is sometimes called `self-modeling' or `shape invariant modeling' after the pioneering work of \citet{lawton1972self}. 
In this line of work, consistency results and/or rates of convergence are obtained, e.g., in \cite{kneip1988convergence, kneip1992statistical, wang1999synchronizing}, and distributional limits are obtained in, e.g., \cite{hardle1990semiparametric, gamboa2007semi, bigot2009estimation, trigano2011semiparametric, kneip1995model, vimond2010efficient}.
We also mention a parallel line of work on the topic motivated by cryo-electron microscopy, where the goal is to recover the shape of a protein (in 3D space) from multiple copies that are frozen and then imaged using an electron microscope \cite{perry2019sample, wang2013exact, perry2018message}. The situation is complicated by the fact that the protein may be in different configurations.

\subsubsection{Exponential families}
An exponential family is --- in our context where the parameterization is arbitrary --- of the form 
\begin{align}
\label{expo}
f_\theta(x) = \exp(\theta^\top T(x) - \Lambda(\theta)) h(x),
\end{align} 
with $\theta \in \Theta$, where $\Theta$ is a convex subset of $\bbR^p$, $T: \bbR^d \to \bbR^p$ is called the sufficient statistic, and $h$ is some density. $\Lambda(\theta)$ is there for normalization. 
The normal family of distributions is an exponential family, but a very special one as it is also a location--scale family. Another important example is the multinomial family, which is often used as a distributional model for a bag-of-words approach to document analysis \cite{carter2009fine, baker1998distributional, mei2007automatic}.
In general, exponential families are particularly important in information geometry \cite{ay2017information, amari2016information}. 



\subsection{Hilbert metrics: $L_2$ and RKHS metrics}
\label{sec:L2}


The $L_2$ metric has historically held a central place in FDA, in particular because it is foundational to PCA --- as already discussed in \secref{PCA}. In the context of a general statistical model as in \eqref{model}, the $L_2$ metric induces the following metric
\begin{align}
\delta_2(f_{\theta},f_{\theta_0})
:= \|f_\theta - f_{\theta_0}\|,
\end{align}
where $\|\cdot\|$ will denote the $L^2(\lambda)$ norm when applied to a function.

\paragraph{Location--scale model}
In a location--scale model as in \eqref{loc scale}, with the base density $f$ being square integrable, the $L_2$ metric induces the following metric: for $\theta = (\theta_{\rm l}, \theta_{\rm s})$ and $\theta_0 = (\theta_{\rm l, 0}, \theta_{\rm s, 0})$,\begin{align}
\Delta_2(\theta,\theta_0)^2 
:= \delta_2(f_{\theta},f_{\theta_0})^2
&= \int \Big[|\theta_{\rm s}|^{-1} f(\theta_{\rm s}^{-1} (x - \theta_{\rm l})) - |\theta_{\rm s, 0}|^{-1} f(\theta_{\rm s, 0}^{-1} (x - \theta_{\rm l, 0}))\Big]^2 \d x \\
&= |\theta_{{\rm s},0}|^{-1} \int \Big[|\theta_{{\rm s},0}\theta_{\rm s}^{-1}| f(\theta_{{\rm s},0} \theta_{\rm s}^{-1} (x - \theta_{\rm l} + \theta_{{\rm l},0})) - f(x)\Big]^2 \d x.
\end{align}

In the case of a {\bf location family} \eqref{loc}, this expression takes the form
\begin{align}
\label{L2 loc}
\Delta_2(\theta,\theta_0)^2 
&= \int \Big[f(x - \theta + \theta_0) - f(x)\Big]^2 \d x.
\end{align}
For the normal location model $\cN(\theta, I)$, this specializes into
\begin{align}
\label{L2 normal}
\Delta_2(\theta, \theta_0) ^2
\propto 1 - \exp(-\|\theta-\theta_0\|^2/4).
\end{align}
For the uniform location model $\Unif(\theta, \theta+1)$ on the real line, it takes the form
\begin{align}
\label{L2 unif}
\Delta_2(\theta, \theta_0)^2
\propto 1 - (1 - |\theta - \theta_0|)_+
= |\theta - \theta_0| \wedge 1.
\end{align}
Clearly, $\Delta_2$ is bounded from above by $2 \|f\|$, and therefore cannot be a Euclidean metric.
In particular, CS cannot be exact (population setting) or consistent (sample setting).

In the case of a {\bf scale family} \eqref{scale}, the induced metric is given by
\begin{align}
\label{L2 scale}
\Delta_2(\theta,\theta_0)^2 
&= |\theta_0|^{-1} \int \Big[|\theta_0\theta^{-1}| f(\theta_0 \theta^{-1} x) - f(x)\Big]^2 \d x.
\end{align}
For the Gamma scale family with shape parameter $k>-1$, where $f(x) \propto x^k \exp(-x) \IND{x > 0}$, this becomes
\begin{align}
\label{L2 Gamma scale}
\Delta_2(\theta,\theta_0) ^2
&\propto \theta^{-1} + \theta_0^{-1} - 4^{k+1} \theta^k \theta_0^k (\theta+\theta_0)^{2k+1}.
\end{align}

\subsubsection{Intrinsic metric}
\label{sec:L2 intrinsic}
Based on \prpref{intrinsic limit} and the accompanying \remref{intrinsic limit}, we know that if $\Theta$ is open and connected --- which is often the case --- and $\Delta_2^2$ defines a topology which is equivalent to the Euclidean topology --- which is the case except in pathological situations --- and, in addition, it is twice continuously differentiable with nonsingular Hessian, then the induced intrinsic metric is Riemannian with tensor given by half the top diagonal block of the Hessian.
 
In some (common) circumstances, that metric tensor can be described as an {\em information matrix} expressed in terms of the derivatives of the statistical model. Indeed, under some conditions on the model $\{f_\theta: \theta \in \Theta\}$, the tensor is given by
\begin{align}
\label{L2 matrix}
\I_2(\theta) := \int \partial f_\theta \partial f_\theta^\top,
\end{align}
where $\partial$ here denotes the differentiation with respect to $\theta$ and the integration is with respect to $\lambda$.
The reader will recognize that this matrix is very similar to the Fisher information matrix, which we introduce later on in \eqref{fisher matrix}. The following result is analogous to \lemref{qmd} given in that subsection. 
(The arguments being the same, we do not provide a separate proof.)

\begin{lemma}
\label{lem:L2 qmd}
In addition to assuming that $f_\theta$ is square integrable for all $\theta$, suppose that $\theta \mapsto f_\theta(x)$ is continuously differentiable for all $x$, that the resulting derivatives are square integrable so that the matrix $\I_2$ above is well-defined, and assume furthermore that $\I_2$ is continuous. 
Under these conditions, for any $\theta_0$ in the interior of $\Theta$,
\begin{align}
\label{L2 qmd}
\Delta_2(\theta, \theta_0)^2
\sim (\theta-\theta_0)^\top \I_2(\theta_0) (\theta-\theta_0),
\quad \text{as } \theta \to \theta_0.
\end{align}
In particular, if the conditions of \prpref{intrinsic limit} are satisfied, then the intrinsic metric induced on $\Theta$ is the Riemannian metric with tensor $\I_2$.
\end{lemma}

\paragraph{Location--scale model}
Consider a square integrable {\bf location family} of densities as in \eqref{loc}. It is straightforward to see, and it is detailed in \cite[Ex 7.8]{van2000asymptotic} and \cite[Cor 12.2.1]{TSH}, that the conditions leading to \eqref{L2 qmd} are satisfied if the base density $f$ is continuously differentiable with compact support. In that case the information matrix \eqref{L2 matrix} is constant, equal to $A := \textstyle\int \nabla f \nabla f^\top$, and an application of the lemma gives
\begin{align}
\Delta_2(\theta,\theta_0)^2 
\sim (\theta-\theta_0)^\top A (\theta-\theta_0),
\quad \text{as } \theta \to \theta_0.
\end{align}
Since $\Delta_2$ is continuous, to show that it is equivalent to the Euclidean metric, it suffices to remark that $\Delta_2(\theta,\theta_0) \to \sqrt{2} \|f\|$ when $\|\theta\| \to \infty$. As $f$ is compactly supported, this is immediate. (When $f$ is not compactly supported, but still square integrable, this can be deduced from taking the Fourier transform and then applying the Riemann--Lebesgue lemma.)
For $\Delta_2^2$ to be twice continuous differentiable, we take $f$ to be thrice continuous differentiable (and still compactly supported). 
In such circumstances, therefore, \prpref{intrinsic limit} applies to affirm that the intrinsic metric induced on $\Theta$ is the one induced by the Euclidean metric given by the constant tensor $A$. By a re-parameterization, we may take $A = I$.

With \prpref{intrinsic limit} operating, we have \eqref{delta approx}, and for the estimation of intrinsic distances based on graph distances to be consistent in the population setting, per \prpref{Isomap n-consistency} we still require the data to be dense in the large-$n$ limit \eqref{dense}. This is only possible if $\Theta$ is bounded for $\Delta_2^2$, or equivalently, for the Euclidean metric. Assuming this is the case, still, Isomap can only be consistent if the intrinsic metric induced on $\Theta$ is Euclidean. Since we know that the intrinsic metric is, here, the one induced by the ambient Euclidean metric, it is Euclidean if and only if $\Theta$ is convex.

In conclusion, Isomap is consistent under \eqref{delta approx} --- and also \eqref{delta uniformly consistent} if we are in the sample setting --- for a location model where the base density is compactly supported and $C^3$, and the parameter space is bounded and convex.
We can contrast this with the fact that CS cannot be exact/consistent here since the metric $\Delta_2$ is, itself, never Euclidean.

\begin{remark}
\label{rem:L2 unif}
The smoothness assumption on the base density is important. Indeed, consider again the case of the uniform location model $\Unif(\theta, \theta+1)$, where the base density is a rather nice, piecewise constant and compactly supported function. Then, based on \eqref{L2 unif}, the intrinsic distance between $\theta_0 < \theta$ is bounded from below as follows 
\begin{align}
\sup_{\theta_0 < \theta_1 < \cdots < \theta_k = \theta}\ \sum_{i=1}^k \sqrt{\theta_i - \theta_{i-1}}
&\ge \sqrt{(\theta-\theta_0) k} \to \infty, \quad k \to \infty,
\end{align}
by considering the grid $\theta_i = (\theta-\theta_0)(i/k)$. 
(We are assuming that $k$ is large enough.)
Hence, the intrinsic distance is the trivial metric: the metric which is infinite between any two distinct points.
Furthermore, the same situation may arise even if the base density is H\"older continuous. Indeed, for $\alpha > 0$, consider $f(x) \propto (1-|x|^\alpha)_+$. 
It can be shown that, for $|\theta-\theta_0| \le 1$, 
\begin{align}
\Delta_2(\theta, \theta_0) \asymp 
\begin{cases}
|\theta-\theta_0|, & \text{if } \alpha > 1/2; \\
|\theta-\theta_0| \sqrt{\log(1/|\theta-\theta_0|)} , & \text{if } \alpha = 1/2; \\
|\theta-\theta_0|^{\alpha+1/2}, & \text{if } \alpha < 1/2.
\end{cases}
\end{align}
And from this it is straightforward to see that the intrinsic metric is the trivial metric whenever $\alpha \le 1/2$, while it is Euclidean when $\alpha > 1/2$.
\end{remark}

In the context of a {\bf scale family} as in \eqref{scale}, the conditions leading to \eqref{L2 qmd} are satisfied if the base density $f$ is continuously differentiable with $\int (\nabla f(x)^\top x)^2 \d x < \infty$. 
This comes from the fact that
\begin{align}
\partial f_\theta(x) 
= - |\theta|^{-1} f(\theta^{-1} x) \theta^{-\top}
- |\theta|^{-1} \theta^{-\top} \nabla f(\theta^{-1} x) x^\top \theta^{-\top}.
\end{align}
After a change of variable, we obtain the following Riemannian metric tensor: at $\theta_0$, it takes the form
\begin{align}
\label{L2 intrinsic scale}
\xi \to |\theta_0|^{-1} \int \trace^2\big((f(x) I + x \nabla f(x)^\top) \theta_0^{-1} \xi \big) \d x.
\end{align}
For \prpref{intrinsic limit} to apply, it is again enough that $f$ be compactly supported and thrice continuously differentiable. Under these conditions, then, the intrinsic metric induced on the parameter space is the Riemannian metric given by the tensor \eqref{L2 intrinsic scale}.
Except in dimension $p=1$ (see below), we do not see a re-parameterization of the model that would make this metric tensor constant, and thus the metric does not appear to be Euclidean.

\paragraph{When is the $L_2$ intrinsic metric Euclidean?}
In general, because we are embedding in a Euclidean space and Isomap is consistent when the intrinsic metric is Euclidean (and the underlying domain is convex), we are interested in knowing when there is a parameterization of the model under consideration that leads to an induced intrinsic metric which is Euclidean. 

Assuming the conditions of \prpref{intrinsic limit} hold, the induced intrinsic metric is Riemannian. Therefore, the question is whether there is a change of variables that makes the metric Euclidean. 
The case where the parameter space has dimension $p=1$: It is well-known that such a re-parameterization exists.

More generally, we may approach this question via the metric tensor: If there is a change of variables that renders the information matrix \eqref{L2 matrix} constant, the metric is Euclidean. 
Thus, consider a diffeomorphism $\varphi$ on $\Theta$. The same model, now parameterized by $\varphi$, has $L_2$ information matrix at $\varphi = \varphi(\theta)$ given by
\begin{align}
D\varphi^{-1}(\varphi(\theta)) \I_2(\theta) D\varphi^{-1}(\varphi(\theta))^\top,
\end{align}
and we want to know if there is a choice of $\varphi$ that makes this constant (and nonzero) --- which may be taken to be the identity matrix without loss of generality. We do not know when this is possible in general, but the case of dimension $p = 1$ is straightforward: the differential equation $\varphi'(\theta)^{-2} \I_2(\theta) = 1$ may be solved by taking $\varphi(\theta) = \int_{\theta_0}^\theta \sqrt{\I_2(\xi)} \d \xi$ where $\theta_0 \in \Theta$ is arbitrary.

The question may also be approached via the curvature tensor: In the particular case of interest here, the curvature tensor is zero if and only if the metric is Euclidean \cite[Th 7.3]{lee2006riemannian}. The situation here is a bit particular because 1) there is a single chart $\theta\in\Theta \to f_\theta$ parameterizes the entire model; and the metric tensor \eqref{L2 matrix} is special in the sense that, at least with some additional smoothness assumptions, $I_2(\theta) = - \int \partial_{\theta\theta} f_\theta$, so that the $(j,k)$ entry is equal to $-\int \partial_{\theta_j} \partial_{\theta_k} f_\theta$.
Even then, the curvature tensor remains very complicated and hard to handle.
Although, under this additional smoothness, we are able to recognize another very particular situation as Euclidean: the metric is Euclidean when there is a parameterization for which the information matrix is diagonal. Indeed, in that case each variable in the parameterization is independent of the others, and can be changed so as to make the tensor constant in that direction.

\subsubsection{Sample setting: RKHS metrics}
We now consider the sample setting of \secref{setting}. When instead of densities we only have available samples to work with, the pairwise $L_2$ distances cannot be directly computed but instead need to be estimated. A natural approach to do so is to use the samples to estimate the densities, and then compute the $L_2$ distances on these estimates. 

Suppose we use kernel density estimation (KDE) based on $\kappa$, so that density $q_j$ is estimated by
\begin{align}
\label{KDE}
\hat q_j(x) := \frac1{m_j} \sum_{i=1}^{m_j} \kappa(x, x_{i,j}),
\end{align}
when $\cS_j := \{x_{1,j}, \dots, x_{m_j,j}\}$.
The squared $L_2$ distance between $q_j$ and $q_k$ is then estimated by plug-in
\begin{align}
\hat\delta_2(q_j, q_k)
:= \delta_2(\hat q_j, \hat q_k)
= \|\hat q_j - \hat q_k\|.
\end{align}
We have
\begin{align}
\|\hat q_j - \hat q_k\|^2
= \textstyle \int \hat q_j^2 + \int \hat q_k^2 - 2 \int \hat q_j \hat q_k.
\end{align}
First,
\begin{align}
\label{int hat q_j}
\int \hat q_j^2(x) \d x 
&= \frac1{m_j^2} \sum_{i=1}^{m_j} \int \kappa(x,x_{i,j})^2 \d x + \frac2{m_j(m_j-1)} \sum_{1 \le i < i' \le m_j} \int \kappa(x,x_{i,j}) \kappa(x,x_{i',j}) \d x.
\end{align}
For appropriate functions $f$ on $\bbR^d$ and $g$ on $\bbR^d \times \bbR^d$, define $g * f(x) = \int g(x,y) f(y) \d y$.
Recalling that the $x_{i,j}$ are iid from $q_j$, by the law of large numbers, assuming that $\int \kappa^2 * q_j < \infty$, we have
\begin{align}
\frac1{m_j} \sum_{i=1}^{m_j} \int \kappa(x,x_{i,j})^2 \d x
\mathop{\longrightarrow}_{m_j\to\infty} \textstyle\int \kappa^2 * q_j,
\end{align}
so that the first term on the right-hand side of \eqref{int hat q_j} tends to zero in probability as $m_j\to\infty$. 
Similarly, if $\int (\kappa * q_j)^2 < \infty$, by the law of large numbers for U-statistics, the second term tends to $\int (\kappa * q_j)^2$.
Hence, all together, we find that $\int \hat q_j^2$ tends to $\int (\kappa * q_j)^2$ in probability as $m_j\to\infty$.
In the same way, under analogous conditions, $\int \hat q_k^2$ tends to $\int (\kappa * q_k)^2$ in probability as $m_k\to\infty$, and $\int \hat q_j \hat q_k$ tends to $\int (\kappa * q_j)(\kappa * q_k)$ under the combined conditions, resulting in
\begin{align}
\label{L2 sample limit}
\hat\delta_2(q_j, q_k)^2 
\longrightarrow \textstyle \int (\kappa * q_j - \kappa * q_k)^2
= \delta_2(\kappa * q_j, \kappa * q_k)^2,
\end{align}
in probability as $m_j, m_k \to \infty$.

So far, we have assumed that the kernel function $\kappa$ remains fixed while the samples increase in size. Commonly, however, the kernel function in \eqref{KDE} involves a bandwidth. In fact, such a kernel function is typically of the form $\kappa(x,y) = b^{-d} {\rm kern}(\|x-y\|/b)$ for some function ${\rm kern}$, and KDE is $L_2$-consistent when $b\to0$ sufficiently slowly, under mild assumptions on ${\rm kern}$ and the density being estimated (in our case, one of the densities in the model). 
When this is the case, 
\begin{align}
\hat\delta_2(q_j, q_k)^2 \to \delta_2(q_j, q_k)^2, 
\quad \text{as } m_j, m_k \to\infty.
\end{align}

\paragraph{RKHS metrics}
Going back to \eqref{L2 sample limit}, and the limit on the right-hand side, an application for the Fubini--Tonelli theorem yields
\begin{align}
\|\kappa * q_j - \kappa * q_k\|^2
= \textstyle \iint \textsc{k}(y,z) (q_j(y)-q_k(y)) (q_j(z)-q_k(z)) \d y \d z,
\label{kappa2}
\end{align}
where $\textsc{k}(y,z) := \int \kappa(x,y)\kappa(x,z) \d x$.
As it turns out, and shown to be true, e.g., in \cite[Eq (9)]{sriperumbudur2010hilbert}, this is the squared distance between $q_j$ and $q_k$ in the metric of the {\em reproducible kernel Hilbert space (RKHS)} defined by the kernel $\textsc{k}$. 
And without letting a bandwidth goes to zero in the large sample limit, it turns out that this is a true metric under some conditions on the kernel function; we refer the reader to the same article \cite{sriperumbudur2010hilbert} and references therein for a thorough discussion. 
In fact, RKHS metrics have been studied for quite some time, in particular as a way to derive nonparametric tests for the two-sample problem \cite{berlinet2011reproducing, gretton2006kernel, smola2007hilbert, szekely2004testing, bakshaev2009goodness, klebanov2005n, zinger1992characterization}.

Let us consider, again, a {\bf location model} with base density $f$ as in \eqref{loc}. We use a translation invariant kernel $\textsc{k}$, meaning, of the form $\textsc{k}(x,y) = \k(x-y)$. 
In that case, the metric takes the following form
\begin{align}
\Delta_\textsc{k}(\theta, \theta_0)^2
&:= \iint \k(y - z) (f(y-\theta+\theta_0)-f(y)) (f(z-\theta+\theta_0)-f(z)) \d y \d z.
\end{align}
We note that $\textsc{k}$ is translation invariant if $\kappa$ itself is translation invariant, in which case it is of the form $\textsc{k}(x,y) = \k * \k(x-y)$, and 
\begin{align}
\label{RHKS Delta k * k}
\Delta_\textsc{k}(\theta, \theta_0)^2
= \|\k * f_\theta - \k * f_{\theta_0}\|^2,
\end{align}
Worth mentioning is the fact that such a metric may be used even on densities that are not square integrable, for example, if $\textsc{k}$ is compactly supported 

\paragraph{RKHS intrinsic metrics}
In the context of an RKHS metric with kernel $\textsc{k}$ as in \eqref{kappa2}, the relevant information matrix is the following
\begin{align}
\I_{\textsc{k}}(\theta) := 
\iint \textsc{k}(y,z) \partial f_\theta(y) \partial f_\theta^\top(z) \d y \d z,
\end{align}
and a result analogous to \lemref{L2 qmd} exists.

In fact, in some situations, the smoothness required of the model above may be ported to the assumed smoothness of the kernel function. For example, in the context of a {\bf location model}, if we take $\textsc{k}$ of the form $\textsc{k}(x,y) = \k * \k (x-y)$ for some compactly supported, even, smooth function $\k: \bbR^d \to \bbR_+$, then starting with \eqref{RHKS Delta k * k}, a Taylor expansion of $\k$ (not of $f$) gives that 
\begin{align}
\Delta_\textsc{k}(\theta, \theta_0)^2
\sim (\theta-\theta_0)^\top A (\theta-\theta_0), &&
A := \iiint f(y) f(z) \nabla\k(x-y) \nabla\k(x-z)^\top \d x\d y\d z.
\end{align}
Thus the induced intrinsic metric is given by the ambient Euclidean metric, regardless of the smoothness of $f$. This is in contrast with the $L_2$ metric for which a single discontinuity in $f$ renders the induced intrinsic metric trivial. 

\paragraph{Computation}
When the dimension $d$ is small, the $L_2$ distance between two densities may be computed, or rather, approximated, by direct numerical integration. And we already described a plug-in approach to its estimation. 

On the other hand, if we work directly with an RKHS metric instead of the $L_2$ metric, it is known that its estimation may be conveniently done by the use of a U-statistic \cite{gretton2006kernel}. And a U-statistic can be efficiently approximated in linear time \cite{lin2010fast}.

\subsection{Divergences: Hellinger and Kullback--Leibler}
\label{sec:hellinger}

 
The {\em Hellinger metric}, defined as
\begin{align}
\delta_{\rm H}(f,g)^2 = \|\sqrt{f} - \sqrt{g}\| = \textstyle\int \big[\sqrt{f}-\sqrt{g}\big]^2,
\end{align}
where the integral, as before, is with respect to $\lambda$.
The (symmetrized) {\em Kullback--Leibler (KL) divergence}, defined as
\begin{align}
\delta_{\rm KL}(f,g) = \textstyle\int \big[f \log (f/g) + g \log (g/f)\big], 
\end{align}
are well-known and particularly popular choices when studying situations where an iid sample is involved. 
In general, consider the (symmetrized) divergence based on a convex function $\psi: (0,\infty) \to \bbR$ 
\begin{align}
\label{delta psi}
\delta_\psi(f,g) = \textstyle\int \big[f \psi (g/f) + g \psi (f/g)\big]. 
\end{align}
When $\psi(t) = \frac12 (\sqrt{t}-1)^2$, we recover the squared Hellinger distance, while $\psi(t) = - \log(t)$ gives the KL divergence. By varying $\psi$, we can obtain other divergences. For example, the total variation is obtained from $\psi(t) = \frac12|t-1|$, while the (symmetrized) {\em $\chi^2$-divergence} results from choosing $\psi(t) = (t-1)^2$.
The usual requirement that $\psi(1) = 0$ implies, via Jensen's inequality, that $\delta_\psi(f,g) \ge 0$, while the assumption that $\psi$ is strictly convex implies that equality holds only when $f = g$ almost everywhere.

\begin{remark}
These are called {\em f-divergences}, and for more background, we refer the reader to the lecture notes by \citet{polyanskiy2014lecture}. This is the only type of divergence that we will consider, so that we simply call them divergences.
\end{remark}

\paragraph{Exponential model}
Divergences are particularly well-suited to deal with exponential families. 
(Some do not behave so well under location models, in particular, when the base density has compact support, where the KL and $\chi^2$ divergences, for example, reduce to the trivial metric.)
Therefore, consider an exponential family as in \eqref{expo}. 
For the Hellinger metric, using the fact that $\E_0[\exp(\xi^\top T(X))] = \exp(\Lambda(\xi))$, we have
\begin{align}
\label{H metric expo}
\Delta_{\rm H}(\theta,\theta_0)^2
:= \delta_{\rm H}(f_\theta, f_{\theta_0})^2
= 2 - 2 \exp\big[\Lambda(\tfrac12(\theta+\theta_0)) - \tfrac12 (\Lambda(\theta)+\Lambda(\theta_0))\big].
\end{align}
For example, in the case of the canonical normal location model $\{\cN(\theta, I) : \theta \in \bbR^d\}$, this becomes 
\begin{align}
\label{H normal}
\Delta_{\rm H}(\theta,\theta_0)^2
= 2 - 2 \exp\big[-\tfrac18 \|\theta-\theta_0\|^2\big].
\end{align}
Regardless of the model, the Hellinger metric is bounded by $\sqrt{2}$, and thus cannot be Euclidean.

For the KL divergence, using the well-known fact that $\E_\theta[T(X)] = \nabla \Lambda(\theta)$, we have
\begin{align}
\Delta_{\rm KL}(\theta,\theta_0)
:= \delta_{\rm KL}(f_\theta, f_{\theta_0})
= (\theta-\theta_0)^\top (\nabla \Lambda(\theta)-\nabla \Lambda(\theta_0)),
\end{align}
which is is a symmetrized {\em Bregman divergence}. 
This corresponds to a Euclidean metric (after taking the square root) if and only if 
\begin{align}
(\theta-\theta_0)^\top (\nabla \Lambda(\theta)-\nabla \Lambda(\theta_0))
= (\theta-\theta_0)^\top A (\theta-\theta_0),
\end{align}
for some positive definite matrix $A$. By differentiating once with respect to $\theta$ and once with respect to $\theta_0$, which is possible since $\Lambda$ is infinitely differentiable in the interior of the parameter space, this identity is seen to be equivalent to the Hessian of $\Lambda$ satisfying $\hess \Lambda(\theta) = A$ for all $\theta$. 
In turn, this is equivalent to $\Lambda$ being of the form $\Lambda(\theta) = \frac12 \theta^\top A \theta$, since $\Lambda(0) = 0$.
In that case, 
\begin{align}
\E_\theta\big[\exp(\xi^\top T(X))\big]
&= \exp\big(\tfrac12 (\xi+\theta)^\top A (\xi+\theta) - \tfrac12 \theta^\top A \theta\big) \\
&= \exp\big(\tfrac12 \xi^\top A \xi + \xi^\top A \theta\big).
\end{align}
On the right-hand side we recognize the moment generating function of the normal distribution with mean $A\theta$ and covariance matrix $A$.
Thus, if we assume that the parameter space $\Theta$ contains a neighborhood of the origin, then necessarily $T(X)$ has that distribution under $\theta$. Equivalently, seen through the sufficient statistic $T$, the model is a normal location model. 
From this we deduce that the metric induced by the KL divergence on an exponential family with natural parameterization is Euclidean if and only if, seen through the sufficient statistics, the model is a normal location family. 

\subsubsection{Intrinsic metric: Fisher}
\label{sec:fisher metric}
The Fisher metric is defined via the Fisher information matrix, here playing the role of metric tensor: assuming a smooth parametric model $\{f_\theta: \theta \in \Theta\}$, the information matrix at $\theta$ is given by
\begin{align}
\label{fisher matrix}
\I(\theta) := \int \frac{\partial f_\theta \partial f_\theta^\top}{f_\theta},
\end{align}
where all the derivatives are with respect to $\theta$.
The usual smoothness assumption is that the model be {\em quadratic mean differentiable (QMD)}, and in such a context, the Fisher information matrix plays a central role in asymptotic statistical theory in the classical setting of a sample growing in size. 

The Fisher metric was originally introduced by \citet{radhakrishna1945information, rao1987differential}. Besides being well-known in information theory and information geometry, it has found applications in computer vision \cite{maybank2020fisher,maybank2019fisher}. It is the metric that \citet{carter2009fine} work with when embedding functional data. 

Here it arises as the intrinsic metric induced by the Hellinger metric on a smooth (QMD) model. Indeed, as in \secref{L2 intrinsic}, based on \prpref{intrinsic limit}, we know that if we are in a situation where $\Theta$ is open and connected, and where $\Delta_2^2$ is such that its topology is equivalent to the Euclidean topology and as a function it is twice continuously differentiable with nonsingular Hessian, then the induced intrinsic metric is Riemannian. As before, it turns out that the metric tensor is given by the Fisher information.

\begin{lemma}[Lem 7.6 in \cite{van2000asymptotic} or Th 12.2.1 in \cite{TSH}]
\label{lem:qmd}
Suppose that $\theta \mapsto f_\theta^{1/2}(x)$ is continuously differentiable for all $x$, that the resulting derivatives are square integrable so that the matrix $\I(\theta)$ above is well-defined for all $\theta$, and assume furthermore that $\I(\cdot)$ is continuous. 
Under these conditions, for any $\theta_0$ in the interior of $\Theta$,
\begin{align}
\label{qmd}
\Delta_{\rm H}(\theta, \theta_0)^2
\sim (\theta-\theta_0)^\top \I(\theta_0) (\theta-\theta_0),
\quad \text{as } \theta \to \theta_0.
\end{align}
\end{lemma}


It turns out that the approximation \eqref{qmd} is also true (up to an unimportant multiplicative factor) for divergences that admit a Taylor expansion of order~2, although only under some additional conditions on the divergence and the model.
Indeed, assume that $\psi$ is twice differentiable, so that a development of order 2 around $t=1$ yields
\begin{align}
\psi(t) = c_1 (t-1) + \tfrac12 c_2 (t-1)^2 + o(t-1)^2, && c_1 := \psi'(1), && c_2 := \psi''(1),
\end{align}
using the fact that $\psi(1) = 0$. Then, as $\theta \to \theta_0$, and again, under some conditions on the remainder term and the model, 
\begin{align}
\delta_\psi(f_\theta, f_{\theta_0}) 
\sim \tfrac12 c_2\, \delta_{\chi^2}(f_\theta, f_{\theta_0}),
\end{align}
using the fact that $\int f_\theta = 1$ for all $\theta$.
And, using the fact that 
\begin{align}
f_{\theta} = f_{\theta_0} + \nabla f_{\theta_0}^\top (\theta - \theta_0) +o(\|\theta - \theta_0\|),
\end{align}
we have, under appropriate regularity conditions,
\begin{align}
\delta_{\chi^2}(f_\theta, f_{\theta_0})
&= \textstyle \int (1/f_\theta + 1/f_{\theta_0}) (f_\theta - f_{\theta_0})^2 \\
&\sim 2 \textstyle \int (1/f_{\theta_0}) (\nabla f_{\theta_0}^\top (\theta - \theta_0))^2 \\
&= 2 (\theta - \theta_0)^\top \I(\theta_0) (\theta - \theta_0),
\end{align}
concluding that
\begin{align}
\label{psi intrinsic limit}
\delta_\psi(f_\theta, f_{\theta_0}) 
\sim c_2\, (\theta - \theta_0)^\top \I(\theta_0) (\theta - \theta_0).
\end{align}

\paragraph{Exponential model}
Rather than providing some technical conditions under which \eqref{psi intrinsic limit} holds, we content ourselves with affirming that they are valid in the context of an exponential family for a broad range of smooth divergences which includes Hellinger, KL, and $\chi^2$.

We now seek to apply \prpref{intrinsic limit} in the context of such a model. We detail the arguments in the case of the Hellinger metric, starting with the closed-form expression for the induced metric derived in \eqref{H metric expo}.
To show that the two metrics are equivalent, since $\Delta_{\rm H}$ is clearly continuous, it is enough to show that 
\begin{align}
R(\theta_n,\theta_0) := \Lambda(\tfrac12(\theta_n+\theta_0)) - \tfrac12 (\Lambda(\theta_n)+\Lambda(\theta_0)) \to 0,
\end{align}
is not possible when $\theta_0\in\Theta$ is fixed while either $\|\theta_n\| \to \infty$ or $\dist(\theta_n,\partial\Theta) \to 0$. 
Let $\eps>0$ be small enough that the closed ball centered at $\theta_0$ of radius $\eps$ is inside $\Theta$, denoted $B$. The function $\theta \mapsto R(\theta,\theta_0)$ being continuous on the corresponding sphere, $\partial B$, and strictly negative because $\Lambda$ is strictly convex, there is $\eta>0$ such that $\min_{\|u\|=1} R(\theta_0+\eps u, \theta_0) \le -\eta$. Also, for any fixed direction $u$, the function $t \mapsto R(\theta_0+t u, \theta_0)$ is non-increasing (in fact, decreasing), because its derivative is $- \frac12 u^\top (\Lambda(\theta_0+(t/2)u) - \Lambda(\theta_0+t u)) < 0$ when $t > 0$, again due to the fact that $\Lambda$ is strictly convex. Hence, it must be the case that $\inf\{R(\theta, \theta_0) : \theta \notin B\} \le -\eta$. And this is what we needed to prove.
Since $\Lambda$ is infinitely differentiable, so is $\Delta_{\rm H}^2$.
Therefore, \prpref{intrinsic limit} applies. Moreover, it is straightforward to verify that \lemref{qmd} also applies, and combined, this confirms that the induced intrinsic metric is the Riemannian metric defined by Fisher information \eqref{fisher matrix}.

\paragraph{When is the Fisher metric Euclidean?}
As is well-known, for an exponential model, the Fisher information matrix corresponds to the Hessian of $\Lambda$, i.e., $\I(\theta) = \hess \Lambda(\theta)$.
And we already saw that $\hess \Lambda$ is constant exactly when the model, seen through the sufficient statistic $T(X)$, is a normal location model. Therefore, this is the only case when the Fisher metric on an exponential family is Euclidean. 

\subsubsection{Sample setting}
In the sample setting of \secref{setting}, we are faced with the problem of estimating the pairwise divergences $\delta_\psi(q_j, q_k)$ based on the samples $\cS_j$ and $\cS_k$, for all $j \ne k$.

At least for the Hellinger metric, a plug-in approach is viable, just like it is for the $L_2$ metric. However, more direct approaches have been proposed, in particular for the estimation of the KL divergence, which is not as well-behaved. This is discussed, e.g., in \cite{perez2008kullback, lee2006estimation, wang2009divergence, wang2005divergence, bu2018estimation}, and in particular in the recent article \cite{zhao2020minimax}, which includes an extensive review of the literature on the problem and shows that a popular direct approach based on comparing the distance to the $k$-nearest neighbor at each location is shown to be minimax optimal.

\subsection{Wasserstein $W_2$ metric}
\label{sec:W2}

Although comparatively complicated to define and handle, the last metric that we consider is nonetheless known for providing an intuitive way of measuring the dissimilarity between distributions. 
The {\em Wasserstein metric} $W_2$ is defined via Kantorovich's formulation of the optimal transport problem: for two densities $f$ and $g$ on $\bbR^d$ with finite second moments,
\begin{align}
\delta_{\rm W}(f,g)^2 := \inf_{\pi \in \Pi(f,g)} \iint \|x-y\|^2 \pi(x, y) \lambda(\d x) \lambda(\d y),
\end{align}
where $\Pi(f,g)$ is the set of densities $\pi$ on $\bbR^d \times \bbR^d$ with marginals $f$ and $g$, here meaning that that 
\begin{align}
\int \pi(x,y) g(y) \lambda(\d y) &= f(x), \text{ for almost all $x$;} \\
\int \pi(x,y) f(x) \lambda(\d x) &= g(y), \text{ for almost all $y$.}
\end{align}
We refer the reader to \cite{santambrogio2015optimal, ambrosio2005gradient, ambrosio2013user} and references therein for background on the optimal transport problem and the resulting Wasserstein metric.
The Wasserstein metric has been the object of much attention in the statistics and machine learning communities in recent years, and some of these developments are surveyed in \cite{panaretos2019statistical}. This more recent enthusiasm is in part due to advances on the computational aspect of the problem \cite{peyre2019computational}.	 

Preceding Kantorovich's, Monge's formulation is based on transport maps, and takes the form of the following optimization problem: for two densities $f$ and $g$ on $\bbR^d$,
\begin{align}
\label{monge}
\inf_{T \in \cT(f,g)} \int \|x-T(x)\|^2 f(x) \lambda(\d x),
\end{align}
where $\cT(f,g)$ is the set of transformations $T : \bbR^d \to \bbR^d$ such that $g$ is the push forward of $f$ by $T$, i.e., $T(X) \sim g$ when $X \sim f$. 
 A celebrated result of \citet{brenier1991polar} (see also \cite[Th 2.26]{ambrosio2013user} or \cite[Th 1.22]{santambrogio2015optimal}) brings these two formulations together. When applied to Lebesgue densities with finite second moments, the result says that the two optimization problems coincide in value and the minimization in \eqref{monge} is achieved by a unique transport map, and that map is the gradient of a convex function. 
In the same context, it is also known that a transport map that is the gradient of a convex function is optimal \cite[Th 2.13]{ambrosio2013user}.

\paragraph{Location--scale model}
Consider a location--scale family of densities as in \eqref{loc scale}. Assume without loss of generality that the base density $f$ has zero mean and let $\Sigma$ denote its covariance matrix.
For $\theta_0$ and $\theta$ in $\Theta$, consider the transport map $x \mapsto \theta_{\rm s}\theta_{\rm s,0}^{-1} x + \theta_{\rm l} - \theta_{\rm s}\theta_{\rm s,0}^{-1}\theta_{\rm l,0}$. This map pushes $f_{\theta_0}$ forward to $f_\theta$ and, being affine, it is the gradient of a convex function. Therefore, it is is optimal and, consequently,
\begin{align}
\Delta_{\rm W}(\theta, \theta_0)^2
:= \delta_{\rm W}(f_\theta, f_{\theta_0})^2
&= \int \big\|x-\big(\theta_{\rm s}\theta_{\rm s,0}^{-1} x + \theta_{\rm l} - \theta_{\rm s}\theta_{\rm s,0}^{-1}\theta_{\rm l,0}\big)\big\|^2 f_{\theta_0}(x) \d x \\
&= \int \big\|\theta_{\rm s,0}x+\theta_{\rm l, 0}-\theta_{\rm s}x-\theta_{\rm l}\big\|^2 f(x) \d x \\
&= \|\theta_{\rm l}-\theta_{\rm l, 0}\|^2 + \trace((\theta_{\rm s}-\theta_{\rm s, 0}) \Sigma (\theta_{\rm s}-\theta_{\rm s, 0})^\top).
\end{align}
We thus see that the induced metric is Euclidean in both location and scale.

\paragraph{Dimension $d=1$: time warping}
For densities on the real line, meaning in dimension $d=1$, the $W_2$ metric is particularly simple: It takes the following explicit form
\begin{align}
\delta_{\rm W}(f,g)^2 = \int_0^1 (F^-(u) - G^-(u))^2 \d u,
\end{align}
where $F^-$ is the quantile function associated with $f$ (i.e., a pseudo-inverse of the distribution function $F(x) := \int_{-\infty}^x f(t) \d t$) and $G^-$ is the quantile function associated with $g$.
Because the Wasserstein metric takes such a simple form, we are able to deal with the general case of a parameter space $\Theta$ which is a subgroup of diffeomorphisms from interval $(a_1,a_2)$ to interval $(b_1,b_2)$, where the endpoints may be infinite. That is to say, if $X\sim f$ then $\theta(X) \sim f_\theta$. In that case, given $\theta, \theta_0 \in \Theta$, consider the transport map $x \mapsto \theta_0\theta^{-1}(x)$. This map pushes $f_\theta$ forward to $f_{\theta_0}$ and is a diffeomorphism from $(b_1,b_2)$ to $(b_1,b_2)$, which in dimension one implies its indefinite integral is a convex function. Therefore, this map is optimal and, consequently, letting $F_\theta(x) := \int_{-\infty}^x f_\theta(t) \d t$,
\begin{align}
\Delta_{\rm W}(\theta, \theta_0)^2
&= \int_0^1 (F_\theta^-(u) - F_{\theta_0}^-(u))^2 \d u \\
&= \int (\theta(F^-(u))-\theta_0(F^-(u)))^2 \d u \\
&= \int (\theta(x)-\theta_0(x))^2 f(x) \d x, \label{W2 1D}
\end{align}
after the change of variables $x = F^-(u)$.
That is, $W_2$ is the $L_2$ metric with weight function the base density $f$. It is therefore a Hilbert metric. 

The metric in this case is so simple as to allow for the treatment of infinite dimensional examples, as exemplified by a {\em time warping}  model \cite[Sec~5.2]{wang2016functional} --- a setting in which the $W_2$ metric has been used \cite{agullo2015parametric}. In line with the literature on the topic, we assume that $f$ is a density on $[0,1]$ and that $\Theta$ is the class of increasing diffeomorphisms of $[0,1]$. 
Although \prpref{CS n-consistency} applies, it is not clear what embedding in finite dimension does for the analyst when the parameter space is effectively infinite dimensional. We show that it is consistent in the asymptotic limit where the embedding dimension $d_e \to \infty$.

\begin{proposition}
\label{prp:time warping}
In the present setting, if $\Delta^{\rm CS}_{d_e}$ denotes the (Euclidean) metric in an embedding of $\Theta$ by CS in dimension $d_e$, then 
\begin{align}
\sup_{\theta, \theta_0 \in \Theta} \big|\Delta^{\rm CS}_{d_e}(\theta,\theta_0) - \Delta_{\rm W}(\theta, \theta_0)\big|
\longrightarrow 0, \quad \text{as } d_e \to \infty.
\end{align}
\end{proposition}

\subsubsection{Intrinsic metric}
In recent work, \citet{li2019wasserstein} develop an asymptotic theory of parameter estimation in the context of the Wasserstein $W_2$ metric that mimics the standard theory, which is instead based on the Hellinger metric, or equivalently, the Kullback--Leibler divergence \cite{van2000asymptotic, TSH}. 

In that context, the Wasserstein information matrix plays the role that the Fisher information matrix \eqref{fisher matrix} plays in the standard theory. It is defined in \cite[Sec~2.2]{li2019wasserstein} in the context of a smooth model of densities $\{f_\theta : \theta \in \Theta\}$ as
\begin{align}
\label{W2 matrix}
\I_{\rm W}(\theta)
:= - \int (\nabla f_\theta) (\Delta_{f_\theta}^{-1} \nabla f_\theta)^\top,
\end{align}
where $\Delta_g^{-1}$ is the inverse of the operator $\Delta_g = {\rm div}(g \nabla)$ applied coordinate-wise. 
More details are provided in \cite{chen2020optimal}. 
In particular, the basis for this derivation is the Benamou--Brenier  formulation of the optimal transport problem
\begin{align}
\label{benamou brenier}
\inf \int_0^1 \|v_t\|_{L^2(\mu_t)} \d t, 
\end{align}
where the infimum is over all $\{(\mu_t, v_t) : t \in [0,1]\}$ where $\mu_t$ is a Borel probability measure on $\bbR^d$ and $v_t: \bbR^d \to \bbR^d$ is a vector field satisfying in the distributional sense to the following {\em continuity equation}
\begin{align}
\dot\mu_t + {\rm div}(v_t \mu_t) = 0,
\end{align}
with $\mu_0 = f \lambda$ and $\mu_1 = g \lambda$. 
Under some conditions, which are for example fulfilled if $f$ and $g$ have compact support, the Benamou--Brenier formulation of the optimal transport problem is equivalent the other two --- Monge's and Kantorovich's --- and the infimum in \eqref{benamou brenier} is achieved by some $(v^*_t)$ whose continuity equation admits a solution of the form $\mu^*_t = h_t \lambda$ (meaning, made of densities) with $(h_t)$ providing a shortest path between $f$ and $g$ in the $W_2$ metric \cite[Ch~8]{ambrosio2005gradient}, \cite[Sec~3]{ambrosio2013user}, \cite[Ch~5]{santambrogio2015optimal}.
The latter justifies using the infinitesimal property encapsulated in the continuity equation to derive the metric tensor \eqref{W2 matrix}.

Thus, in examples where \prpref{intrinsic limit} applies, which is the case, for example, if Wasserstein metric induced on the parameter space is as in \remref{intrinsic limit}, the intrinsic metric induced by $W_2$ is Riemannian with tensor given by \eqref{W2 matrix}. 

Beyond that, we do not know when there is a parameterization for which this metric is induced by the ambient Euclidean metric. (It is obviously so, for example, in a location model, as the induced metric itself is Euclidean.) 

\subsubsection{Sample setting}
In the sample setting of \secref{setting}, the analyst is effectively confronted with the problem of estimating the pairwise divergences $\delta_{\rm W}(q_j, q_k)$ based on the samples $\cS_j$ and $\cS_k$, for all $j \ne k$, as as preliminary for embedding, either via CS or Isomap.
The problem of estimating the Wasserstein metric is reviewed in \cite[Sec~8.4]{peyre2019computational}, where the default strategy is said to be estimating $\delta_{\rm W}(q_j, q_k)$ by the $W_2$ distance between the empirical distributions given by the samples $\cS_j$ and $\cS_k$. Regularized variants of this empirical estimate have been proposed, e.g., in  \cite{chizat2020faster, cuturi2013sinkhorn}.

\section{Discussion}
\label{sec:discussion}

As one of the only excursions into functional multidimensional scaling --- which could be given the acronym FMDS --- the present manuscript is only meant to provide a preliminary idea of what is possible. Certainly, a number of other expeditions come to mind.

\subsection{Regression models}
We focused on a setting where the data points are densities or samples from densities (\secref{setting}). While we only considered the emblematic examples of location--scale families and exponential families (\secref{models}), the setting also includes regression models. 

For example, a linear model built on basis functions $\phi_1, \dots, \phi_p$ is of the form $\phi_\theta(x) := \theta_1 \phi_1(x) + \cdots + \theta_p \phi_p(x)$, defined for $\theta = (\theta_1, \dots, \theta_p) \in \bbR^p$. 
Assuming the design is random, a distributional model for observations based on such a model is of the form 
\begin{align}
\label{regression}
f_\theta(x,y) = g(x) h(y - \phi_\theta(x)),
\end{align}
where $g$ is a density on $\bbR^{d-1}$ and $h$ is a density on $\bbR$. Another way to write this model is as the additive regression model where the response variable satisfies
\begin{align}
Y = \phi_\theta(X) + Z,
\end{align}
where the design variable $X$ and the noise variable $Z$ are independent with $X \sim g$ and $Z \sim h$.
For example, in a linear model with uniform design on the unit interval and normal noise, $g$ is the uniform density on $[0,1]$ and $h$ is a normal density with zero mean. 
  
With a small adaptation, we obtain a single-index model, by considering, instead, $\varphi_\theta(x) := \psi(\theta_1 \phi_1(x) + \cdots + \theta_p \phi_p(x))$, for some function $\psi: \bbR \to \bbR$. 
 
General linear models are obtained by taking, instead, 
\begin{align}
f_\theta(x,y) = g(x) h_{\phi_\theta(x)}(y),
\end{align} 
where $g$ is a density on $\bbR^{d-1}$ and $\{h_\gamma\}$ is an exponential family of densities on $\bbR$. If the exponential family is the Bernoulli family, then we may obtain the (binary) logistic model, or the probit model, with an appropriate parameterization of the family.

We can even accommodate fixed designs, although with a more substantial departure from the general model.

In the FDA setting that we consider, we expect multiple values of the parameter $\theta$ represented in the data, which in effect means that the setting is that a mixture of regression models. This is a topic with a good amount of literature \cite{de1989mixtures, viele2002modeling}. Recent methodology is reviewed in \cite{kwon2020converges}, where an EM approach is considered and analyzed.

While the models above fit squarely within the general setting that we consider, in some situations the analyst may not want to assume that the design is the same throughout or that the noise distribution is the same throughout. Indeed, in regression, the focus is most typically on the conditional distribution of the response variable given the predictor variable --- in which case the design distribution plays no role --- or even its conditional mean --- in which case neither the design distribution nor the noise distribution plays role. In that case, the analyst may be more comfortable with a model of the form
\begin{align}
f_{\theta, \alpha, \sigma}(x,y) = g_\alpha(x) h_\sigma(y - \phi_\theta(x)).
\end{align}
Such models do not fall within the purview of our study, as it is in general the case that $f_{\theta, \alpha_1, \sigma_1} \ne f_{\theta, \alpha_2, \sigma_2}$, and if the focus is on $\theta$ while $\alpha$ and $\sigma$ are considered to be nuisance parameters, then a metric or divergence used to define a distance between two densities in that family may not be appropriate. 
Instead, what the situation calls for is the use of pseudo-metrics or pseudo-divergences. Indeed, ideally, we would like a dissimilarity $\delta$ such that $\delta(f_{\theta, \alpha_1, \sigma_1}, f_{\theta, \alpha_2, \sigma_2}) = 0$. In principle, our discussion generalizes to a large extend to such dissimilarities. However, it appears to us more natural to simply base the inference not on the densities themselves, but on the conditional means (or estimates) if these are really the objects of interest. In that case, the setting is really that of functions that are not necessarily densities. Our treatment of the $L_2$ metric, for example, extends to such a setting without much effort.

\subsection{Other embedding methods}
We focused on classical scaling in its pure form, that is, applied directly to the available or computed dissimilarities, or in its MDS-D \cite{kruskal1980designing, shang2003localization} or Isomap form \cite{silva2002global,Tenenbaum00ISOmap}, that is, applied to the (estimated) intrinsic dissimilarities. 

However, other methods for MDS are available. An emblematic method which is related to the ones considered here, is the localized variant of MDS-D proposed by \citet{shang2004improved}, which in the realm of manifold learning corresponds to a localized variant of Isomap --- an approach suggested in multiple places \cite{rosman2010nonlinear, Schwartz2019, arias2022supervising}. 
In a functional setting, too, such methods may perform better in some situations.


\subsection{Modes of variation}
An important part of applying PCA are the principal directions. These are ordered from the direction of highest variability to the direction of lowest variability. They are typically used as a basis for a linear representation of the underlying model (or point set in the case of data points in space).
In the context of FDA, the principal directions are functions known as the modes of variation and can be plotted for data exploration. 
An embedding via classical scaling or any other method for MDS, does not automatically provide modes of variation.  

One way to obtain modes of variation is to perform PCA in the embedding space and then map these back in function space, as proposed by \citet{chen2012nonlinear}. Note that this requires mapping an arbitrary point or direction in Euclidean space (where the embedding takes place) to function space (where the data reside). This out-of-sample extension is performed in \cite{chen2012nonlinear} by local averaging. (Note that this is a weighted pointwise average of data functions, where the weights are functions of the distances between the embedded points.)

Another way to obtain modes of variation is via {\em geodesic PCA}, which is a form of PCA adapted to data points living on a manifold \cite{fletcher2004principal, sommer2010manifold}. The use of geodesic PCA in FDA was recently proposed by \citet{bigot2017geodesic}, although only in the very special case of dimension $d=1$; see the broader discussion in the survey paper \cite{bigot2020statistical}.
We note that geodesic PCA is in principle appropriate when using one of the metrics or divergences considered in \secref{examples} to embed a smooth model, since any of them induces a differentiable manifold structure in that case. 

\subsection{Ordinal embedding}
Multidimensional scaling comes in two main forms. We focused entirely on one of them, the {\em metric} variant. Of equal, or even more, importance is the {\em non-metric} or {\em ordinal} variant, where the dissimilarities, as they are, are not believed to be Euclidean. (In fact, this is already the case when using, for example, the KL divergence. Even in its intrinsic form, and when the underlying model is a location family of densities, we still need to take the square root to obtain a Euclidean metric.) 
In real life, a non-metric approach may be called for in a situation where someone, perhaps with some expertise, is asked to decide whether some images or art pieces or other `complex things' --- which would warrant a functional modeling --- are closer or farther when compared with each other. 

The idea, roughly speaking, is to determine a monotonic transformation that makes the dissimilarities as Euclidean as possible, in preparation for an embedding in a Euclidean space --- so that what the analyst sees after embedding is a more faithful representation of the data.
This problem has a long history and several methods have been proposed \cite{young1987multidimensional, kruskal1964multidimensional, shepard1962analysisI, shepard1962analysisII, borg2005modern}. 

We simply note that, in some of the examples detailed in \secref{examples}, the induced metric is an increasing function of the Euclidean metric, and this would imply that ordinal embedding is exact --- at least if the ordinal embedding problem is solved exactly \cite{arias2017some,klein}. This includes the normal location model under the $L_2$ and Hellinger metrics \eqref{L2 normal} and \eqref{H normal}, and also the non-smooth uniform location model under the $L_2$ metric \eqref{L2 unif} (which is identical under the Hellinger metric). 
Even when the metric is not an increasing function of the Euclidean metric, as long as this is true when fixing one of the end points, ordinal embedding restricted to triple comparisons is exact.

\section{Technical details}
\label{sec:proofs}

We use $\|\cdot\|_2$ to denote the Hilbert-Schmidt norm of a bounded operator that acts on a Hilbert space. In particular, $\|A\|_2$ is the Frobenius norm of $A$ if $A$ is a matrix, and $\|\fK\|_2^2 = \iint \kappa^2 (s,t) \lambda(\d s) \lambda(\d t)$, if $\fK$ is an integral operator with kernel $\kappa$, i.e., $\fK q(\cdot)= \int \kappa(s,\cdot) q(s) \lambda(\d s)$, $q\in L_2(\lambda)$.


\paragraph{Proof of \prpref{PCA stress}}
Define $\fK: q\mapsto \int \kappa(s,\cdot)q(s) \d\lambda(s)$, which is the integral operator corresponding to $\kappa$. By Mercer's theorem, we can write
\begin{align}
\kappa(s,t) = \sum_{k=1}^\infty \lambda_k \phi_k(s) \phi_k(t), 
\end{align}
where $\lambda_1\ge \lambda_2\ge \cdots$ are the (non-negative) eigenvalues of $\fK$, and $\phi_k$ is the orthonormal eigenfunction associated with $\lambda_k$. Without loss of generality, we assume that the $q_i$ have been centered, that is, $\bar q \equiv 0$, so that $\kappa(s,t) := \sum_{i=1}^n q_i(s) q_i(t)$ and $\fK \phi_k = \sum_{i=1}^n q_i \<q_i, \phi_k\> = \lambda_k\phi_k$, for all $k\ge1$.

Consider an arbitrary orthonormal basis $\{\ell_k\}_{k=1}^\infty$ of $L_2(\lambda)$ such that $q_i = \sum_{k=1}^{\infty}\< q_i, \ell_k \> \ell_k$ for $i=1,\dots,n$. 
The orthogonal projection of $q_i$ onto the space spanned by $\{\ell_k\}_{k=1}^{d_e}$ is $\sum_{k=1}^{d_e}\< q_i, \ell_k \> \ell_k,$ and the induced point in $\bbR^{d_e}$ is $p_i = (\< q_i, \ell_1 \>,\dots, \< q_i, \ell_{d_e} \>)^\top$. 
Below we will show that the stress in \eqref{stress} is minimized when $\{p_1,\dots,p_n\}$ is returned by PCA, that is, when $\ell_k = \phi_k$, $k=1,\dots,d_e$. 

Note that
\begin{align}
\delta_{ij}^2 = \< q_i - q_j, q_i - q_j \> = \sum_{k=1}^{\infty} \< q_i - q_j, \ell_k \>^2 \ge \sum_{k=1}^{d_e} \< q_i - q_j, \ell_k \>^2 = \|p_i - p_j\|^2.
\end{align}
This yields
\begin{align}
\sum_{i, j} \big|\delta_{ij}^2 - \|p_i - p_j\|^2 \big| & = \sum_{i, j} \sum_{k=d_e+1}^{\infty} \< q_i - q_j, \ell_k \>^2 \\
& = \sum_{i, j} \sum_{k=d_e+1}^{\infty} \sum_{r=1}^{\infty} \<q_i - q_j,  \phi_r \>^2 \< \phi_r, \ell_k \>^2 \\
& = \sum_{k=d_e+1}^{\infty} \sum_{r=1}^{\infty} \< \phi_r, \ell_k \>^2 \sum_{i, j}  \<q_i - q_j,  \phi_r \>^2.
\end{align}
Notice that 
\begin{align}
\sum_{i, j}  \<q_i - q_j,  \phi_r \>^2 = 2 n\sum_{i}  \<q_i,  \phi_r \>^2 - \sum_{i, j}  \<q_i , \phi_r \> \<q_j , \phi_r \>= 2n\<\fK\phi_r,  \phi_r \>,
\end{align}
where we have used the assumption $\bar q \equiv 0$ and the definition of $\fK$. Hence
\begin{align} 
\sum_{i, j} \big|\delta_{ij}^2 - \|p_i - p_j\|^2 \big| & = 2n \sum_{r=1}^{\infty} \<\fK\phi_r,  \phi_r \>\sum_{k=d_e+1}^{\infty}\< \phi_r, \ell_k \>^2  \\
& = 2n \sum_{r=1}^{\infty} \lambda_r \sum_{k=d_e+1}^{\infty} \< \phi_r, \ell_k \>^2.
\end{align}

Denote $h_r=\sum_{k=d_e+1}^{\infty} \< \phi_r, \ell_k \>^2$. Note that $0\le h_r \le 1$ and $\sum_{r=1}^{\infty} (1-h_r) = \sum_{r=1}^{\infty}\sum_{k=1}^{d_e} \< \phi_r, \ell_k \>^2 = d_e.$ Hence the above function is minimized when $h_1=\cdots=h_{d_e}=0$, and $h_{d_e+1}=\cdots=1$, which is achieved when $\ell_k=\phi_k$ for all $k\ge 1$, that is, $p_i$'s are returned by PCA.

\paragraph{Proof of \prpref{PCA consistency}}
We continue using the notation defined in the proof of \prpref{PCA stress}. Let $\hat \fK$ be the integral operator with kernel $\hat\kappa$, i.e., $\hat\fK q(\cdot)= \int \hat\kappa(s,\cdot) q(s) \lambda(\d s)$, $q\in L^2$. Let $\hat\lambda_1\ge \hat\lambda_2\ge \cdots$ be the eigenvalues of $\hat\fK$ with corresponding orthonormal eigenfunctions $\hat\phi_1, \hat\phi_2, \dots$. 

Let $d_c$ be the smallest $d$ such that $d\ge d_e$ and $\lambda_{d+1}>\lambda_d$. Note that $d_c$ is finite because $\fK$ is finite-dimensional. Suppose that $\zeta_1>\cdots>\zeta_{c}=\lambda_{d_e}>\zeta_{c+1}>\cdots$ are the distinct eigenvalues of $\fK$, where $\zeta_i$ has multiplicity $a_i\ge1$. Note that $d_c = a_1+\cdots+a_c$. Let $\cO$ be the group of $d_c \times d_c$ diagonal block orthogonal matrices, where the diagonal blocks are $a_i \times a_i$ orthogonal matrices, $i=1,\dots,c$. 

For any $R = (r_{ij})\in\cO$, let $\phi_i^R = \sum_{j=1}^{d_c} r_{ij} \phi_j$, $i=1,\dots,d_c$. 
For any $q$, define
\begin{equation}
\|q\|_\kappa^2 = \< \fK q, q\> = \sum_{i=1}^n \<q_i - \bar q, q \>^2.
\end{equation}
Let $p_i^* = (\< q_i - \bar q, \phi_1 \>,\dots, \< q_i - \bar q, \phi_{d_c} \>)^\top$, $\hat p_i^* = (\< \hat q_i - \hat{\bar q}, \hat\phi_1 \>,\dots, \< \hat q_i - \hat{\bar q}, \hat\phi_{d_c} \>)^\top$, both in $\bbR^{d_c}$. For any $R\in\cO$, we have
\begin{align}
&\sum_{i=1}^n \|R p_i^* - \hat p_i^* \|^2 \\
&=\sum_{i=1}^n \sum_{k=1}^{d_c} (\< q_i - \bar q, \phi_k^R \> - \< \hat q_i - \hat{\bar q}, \hat \phi_k \>)^2 \\
&=\sum_{i=1}^n \sum_{k=1}^{d_c} (\< q_i- \bar q, \phi_k^R \> -  \< q_i- \bar q, \hat \phi_k \>  + \< q_i- \bar q, \hat\phi_k \> -  \< \hat q_i - \hat{\bar q}, \hat \phi_k \>)^2\\
& \le 2\sum_{i=1}^n \sum_{k=1}^{d_c} \< q_i- \bar q, \phi_k^R - \hat \phi_k \>^2  + 2\sum_{i=1}^n \sum_{k=1}^{d_c} \< q_i - \hat q_i - ( \bar q - \hat{\bar q}), \hat\phi_k \>^2 \\
& \le 2\sum_{k=1}^{d_c} \|\phi_k^R - \hat \phi_k \|_\kappa^2 + 2\sum_{i=1}^n \| q_i - \hat q_i - ( \bar q - \hat{\bar q})\|^2 \\
%
%
& \le 2\lambda_1\sum_{k=1}^{d_c} \|\phi_k^R - \hat \phi_k \|^2  + 2\sum_{i=1}^n \| q_i - \hat q_i\|^2.
\end{align}
Therefore
\begin{align}
\label{pstar bound}
\min_{R \in \cO}\sum_{i=1}^n \|R p_i^* - \hat p_i^* \|^2 \le 2\lambda_1 \min_{R \in \cO}\sum_{k=1}^{d_c} \|\phi_k^R - \hat \phi_k \|^2  + 2\sum_{i=1}^n \| q_i - \hat q_i\|^2.
\end{align}

For $\Phi_i=\{\phi_k : \lambda_k = \zeta_i\}$ and $\hat\Phi_i=\{\hat\phi_k : \lambda_k = \zeta_i\}$, let $\Theta(\Phi_i, \hat \Phi_i)$ be the $a_i \times a_i$ diagonal matrix whose diagonal entries are the principal angles between the subspaces spanned by $\Phi_i$ and $\hat \Phi_i$. It follows from standard arguments (\cite[Sec~II.4]{stewart1990matrix}) that 
\begin{align}
\min_{R \in \cO}\sum_{k=1}^{d_c}  \| \phi_k^R - \hat \phi_k \|^2 = \sum_{i=1}^c \|2\sin(\tfrac{1}{2} \Theta(\Phi_i, \hat \Phi_i))\|_2^2,
\end{align}
where the $\sin$ function is applied entrywisely. 
%
Using the elementary inequality $\sin(\alpha/2) \le \sin(\alpha)/\sqrt{2}$, 
\begin{align}
\label{phiR bound}
\min_{R \in \cO}\sum_{k=1}^{d_c}  \| \phi_k^R - \hat \phi_k \|^2 \le 2 \sum_{i=1}^c\| \sin(\Theta(\Phi_i, \hat \Phi_i))) \|_2^2 \le 16 \sum_{i=1}^c \frac{ \|\fK - \hat{\fK}\|_2^2}{\min_{j\neq i} |\zeta_j - \zeta_i|^2},
\end{align}
where the last inequality is a consequence of the Davis-Kahan $\sin(\Theta)$ theorem \cite{yu2015useful}. 
%
%
%

Let $u_i = q_i - \bar q$ and $\hat u_i = \hat q_i - \bar{\hat{q}}$. We have
\begin{align}
\|\fK - \hat{\fK}\|_2^2 & = \iint \Big(\sum_{i=1}^n u_i(s) u_i(t) - \hat u_i(s) \hat u_i(t) \Big)^2 \d\lambda(s)\d\lambda(t) \\
& \le n \sum_{i=1}^n\iint  \big(u_i(s) u_i(t) - \hat u_i(s) \hat u_i(t) \big)^2 \d\lambda(s)\d\lambda(t) \\
& \le 3n \Big( \sum_{i=1}^n \|\hat u_i - u_i\|^4 + 2 \sum_{i=1}^n \|u_i\|^2 \|\hat u_i - u_i\|^2\Big)\\
& \le 3n \Big( \sum_{i=1}^n \|\hat u_i - u_i\|^2  + 2 \max_i \|u_i\|^2 \Big) \sum_{i=1}^n \|\hat u_i - u_i\|^2.
\end{align}
%
%
%
%
%
%
%
%
%
%
%
Note that 
\begin{align}
\max_i\|u_i \| 
\le \max_i\|q_i\| +  \frac{1}{n} \sum_{j=1}^n \|q_j\| \le 2\max_i\|q_i\| .
\end{align}
Therefore, we have for some constant $C_1>0$ depending only on $\max_i\|q_i\|$ such that
\begin{align}
\|\fK - \hat{\fK}\|_2^2 \le C_1  n  \sum_{i=1}^n \| q_i - \hat q_i\|^2 \Big( 1 + \sum_{i=1}^n \| q_i - \hat q_i\|^2\Big),
\end{align}
This, together with \eqref{pstar bound} and \eqref{phiR bound}, yields
\begin{align}
\label{Rminimum}
\min_{R \in \cO}\sum_{i=1}^n \|R p_i^* - \hat p_i^* \|^2 \le C_2  n\sum_{i=1}^n \| q_i - \hat q_i\|^2\Big( 1 + \sum_{i=1}^n \| q_i - \hat q_i\|^2\Big), \text{ where } C_2 :=  \sum_{i=1}^c \frac{32\lambda_1 C_1}{\min_{j\neq i} |\zeta_j - \zeta_i|^2} + 2.
\end{align}

Define $G_i=(\< \phi_j, \hat\phi_\ell \>)_{j,\ell \in \{k:\lambda_k = \zeta_i\}}$, and $G=\diag(G_1,\dots,G_c)$. Let $F=(f_{ij})=UV^\top\in \bbR^{d_c\times d_c}$, where $U$ and $V$ are $d_c\times d_c$ orthogonal matrices obtained from the SVD decomposition $G = UDV^\top$. Define $\tilde \phi_i = \sum_{j=1}^{d_c} f_{ij} \phi_j$, $i=1,\dots,d_c$. Note that $\{\tilde \phi_i\}_{i=1}^{d_c}$ is an orthonormal basis of the space spanned by $\{\phi_i\}_{i=1}^{d_c}$, since $F$ is an orthogonal matrix. In particular, $\fK \tilde\phi_i = \lambda_i \tilde\phi_i$, $i=1,\dots,d_c$. It is known (see \cite[Sec~II.4]{stewart1990matrix}) that the above minimum on the left side of \eqref{Rminimum} is achieved when $R=F$, that is, when $R p_i^* = \tilde p_i^*: = (\< q_i - \bar q, \tilde\phi_1 \>,\dots, \< q_i - \bar q, \tilde\phi_{d_c} \>)^\top$. Let $\tilde p_i = (\< q_i - \bar q, \tilde\phi_1 \>,\dots, \< q_i - \bar q, \tilde\phi_{d_e} \>)^\top$, $i=1,\dots,n$, which is returned by a $d_e$-dimensional PCA projection based on $q_1, \dots, q_n$. 
It then follows from \eqref{Rminimum} that
\begin{align}
\sum_{i=1}^n\|\tilde p_i - \hat p_i \|^2 \le \sum_{i=1}^n\|\tilde p_i^* - \hat p_i^* \|^2 \le C_2  n\sum_{i=1}^n \| q_i - \hat q_i\|^2 \Big( 1 + \sum_{i=1}^n \| q_i - \hat q_i\|^2\Big).
\end{align}

\paragraph{Proof of \prpref{PCA n-consistency}}
Without loss of generality, we assume that $\E Q \equiv 0$, which will simplify our calculations. 
Let $K(s,t)=\Cov(Q(s), Q(t))$ be the covariance kernel of the process $Q$. By Mercer's theorem, 
\begin{equation}
K(s,t) = \sum_{j=1}^\infty \lambda_j \phi_j(s) \phi_j(t),
\end{equation}
where $\lambda_1\ge \lambda_2\ge \cdots$ are the (nonnegative) eigenvalues of the integral operator $\fK$ with kernel $K$, and $\phi_k$ is the orthonormal eigenfunction associated with $\lambda_k$. 
Below we show that the projection $\pi: q \mapsto (\<q,\phi_1\>,\dots,(\<q,\phi_{d_e}\>)$ minimizes the expected stress in \eqref{expected_stress}. 

Let $\{\ell_k\}_{k=1}^\infty$ be an orthonormal basis of $L^2(\lambda)$, so that for any $Q$, 
\begin{align}
\E \Big\| Q - \sum_{k=1}^{\infty}\< Q, \ell_k \> \ell_k \Big\|^2 = 0.
\end{align}
The orthogonal projection of $Q$ onto the space spanned by $\{\ell_k\}_{k=1}^{d_e}$ is $\pi(Q):=(\< Q, \ell_1 \> ,\dots,\< Q, \ell_{d_e} \>)$. We have
\begin{align}
\E\big|\|Q-Q'\|^2 - \|\pi(Q) -\pi(Q')\|^2 \big| & = \E \sum_{k=d_e+1}^{\infty} \< Q - Q', \ell_k \>^2 \\
& = \E \sum_{k=d_e+1}^{\infty} \sum_{r=1}^{\infty} \<Q - Q',  \phi_r \>^2 \< \phi_r, \ell_k \>^2\\
& = \sum_{k=d_e+1}^{\infty} \sum_{r=1}^{\infty} \< \phi_r, \ell_k \>^2 \E  \<Q - Q',  \phi_r \>^2 \\
& = 2 \sum_{r=1}^{\infty} \lambda_r \sum_{k=d_e+1}^{\infty} \< \phi_r, \ell_k \>^2.
\end{align}
where we have used the fact that 
\begin{equation}
\E [(Q(s) - Q'(s))(Q(t) - Q'(t))] = 2K(s,t).
\end{equation}
The rest of the proof to show that $\pi$ minimizes the expected stress in \eqref{expected_stress} follows the same argument as in the proof of \prpref{PCA stress}. 
Due to the possible multiplicity of any $\lambda_k$, $1\le k\le d_e$, the choice of the orthonormal eigenfunction corresponding to $\lambda_k$ may not be unique. The specific ones used for $\pi$ are denoted by $\phi_k=\phi_k^\pi$, $1\le k\le d_e$, and the set of projections in the form of $\pi: q \mapsto (\<q,\phi_1^\pi\>,\dots,(\<q,\phi_{d_e}^\pi\>)$ is denoted by $\Pi_0$. 

Define
\begin{align}
\hat K(s,t) := \frac1n\sum_{i=1}^n Q_i(s) Q_i(t) .
\end{align}
Let $\hat\lambda_1\ge \hat\lambda_2\ge \cdots$ be the eigenvalues of the integral operator $\hat\fK$ with kernel $\hat K$, and $\hat\phi_k$ be the orthonormal eigenfunction associated with $\hat\lambda_k$. The PCA orthogonal projection based on $Q_1,\dots,Q_n$ is $\hat\pi: q \mapsto (\<q,\hat \phi_1\>,\dots,(\<q,\hat \phi_{d_e}\>)$. 
Following the same arguments as in the proof of \prpref{PCA consistency}, and using the Cauchy--Schwarz inequality, for any $q\in L_2(\lambda)$, we have
\begin{align}
\min_{\pi\in\Pi}\|\pi(q) - \hat\pi(q)\|^2 & \le \min_{\pi\in\Pi_0} \sum_{k=1}^{d_e} \<q, \phi_k^\pi - \hat \phi_k\>^2 \\
& \le \|q\|^2 \min_{\pi\in\Pi_0}\sum_{k=1}^{d_e} \|\phi_k^\pi - \hat \phi_k \|^2  \\
& \le \Big(\sum_{i=1}^{c}  \frac{16 \lambda_1 }{\min_{j\neq i} |\zeta_j - \zeta_i|^2} \Big) \|q\|^2 \|\fK - \hat{\fK}\|_2^2= : C_1 \|q\|^2 \|\fK - \hat{\fK}\|_2^2,
\end{align}
where $\zeta_1>\zeta_2>\cdots$ are defined in the proof of \prpref{PCA consistency}. By \cite[Thm 8.1.2]{hsing2015theoretical}, as $n\to\infty$, $\|\fK - \hat{\fK}\|_2^2 \to 0$, a.s., which then implies that $\min_{\pi\in\Pi}\|\pi(q) - \hat\pi(q)\|^2 \to 0$, a.s.

Now assume that $\E \|Q\|^4 < \infty$. By \cite[Th 2.5]{horvath2012inference}, we have
\begin{align}
\E \|\fK - \hat{\fK}\|^2 \le n^{-1} \E \|Q\|^4,
\end{align}
which implies
\begin{align}
\E\min_{\pi\in\Pi} \|\pi(q) - \hat\pi(q)\|^2 \le n^{-1} C_1\|q\|^2\E \|Q\|^4 . 
\end{align}


\paragraph{Proof that PCA and CS return the same output}
Let $a_k = (\< q_1, \phi_k\>,\dots,\< q_n, \phi_k\>)^\top$. Note that $\|a_k\|^2 = \lambda_k$. It is easy to verify that
\begin{align}
B a_k & = (\sum_{j=1}^n \<q_1, q_j\> \< q_j, \phi_k\>,\dots, \sum_{j=1}^n \<q_n, q_j\> \< q_j,  \phi_k\>)^\top \\
& = (\lambda_k \sum_{j=1}^n \<q_1, \phi_k\>,\dots, \lambda_k \sum_{j=1}^n \<q_n, \phi_k\>)^\top = \lambda_k a_k,
\end{align}
which implies that $u_k = \lambda_k^{-1/2} a_k$ and $\nu_k=\lambda_k$, for $k=1,\dots,d_e$. Hence
\begin{align}
(\sqrt{\nu_1} u_{i,1}, \dots, \sqrt{\nu_{d_e}} u_{i,d_e}) = (\< q_i, \phi_1 \>,\dots, \< q_i, \phi_{d_e} \>),
\end{align}
which is exactly the same point in $\bbR^{d_e}$ mapped by PCA from $q_i$.
%
%


\paragraph{Proof of \prpref{CS consistency}}
%
We have
\begin{align}
\sum_{i=1}^n\|p_i - \hat p_i\|_2^2&= \sum_{i=1}^n \sum_{k=1}^{d_e} ( \sqrt{\nu_k^+}u_{i,k} - \sqrt{\hat \nu_k^+}\hat u_{i,k} )^2 \\
& \le 2 \sum_{i=1}^n \sum_{k=1}^{d_e} ( \sqrt{\nu_k^+}u_{i,k} - \sqrt{\nu_k^+}\hat u_{i,k} )^2 + 2 \sum_{i=1}^n \sum_{k=1}^{d_e} ( \sqrt{\nu_k^+} \hat u_{i,k} - \sqrt{\hat \nu_k^+}\hat u_{i,k} )^2\\
& \le 2 \sqrt{\nu_1} \sum_{k=1}^{d_e} \| u_{k} - \hat u_{k} \|^2  + 2 \sum_{k=1}^{d_e} ( \sqrt{\nu_k^+}  - \sqrt{\hat \nu_k^+})^2.
\end{align}

By the Davis-Kahan $\sin(\Theta)$ theorem \cite{yu2015useful},
\begin{align}
\min_{\{u_1,\dots,u_{d_e}\}}\sum_{k=1}^{d_e}\| u_{k} - \hat u_{k} \|^2 \le \sum_{i=1}^c\frac{8 \|B - \hat B\|_2^2}{\min_{j\neq i} |\zeta_j - \zeta_i|^2} =: C_1 \|B - \hat B\|_2^2,
\end{align}
where $\zeta_1>\zeta_2>\cdots>\zeta_b$ are the distinct eigenvalues of $B$ for some $b\ge1$ with $\zeta_c = \nu_{d_e}$, and $\zeta_{b+1}=-\infty$. 
On the left side of the above inequality, the minimum is over all the orthonormal eigenvectors of $B$ corresponding to its top $d_e$ eigenvalues. 

We also have
\begin{align}
\label{nusqrt}
|\sqrt{\nu_k^+} - \sqrt{\hat \nu_k^+} | \le 
\begin{cases}
|\nu_k - \hat \nu_k |/\sqrt{\nu_k}, & \text{if } \nu_k > 0; \\
\sqrt{|\nu_k - \hat \nu_k |}, & \text{if } \nu_k \le 0.
\end{cases}
\end{align}
Applying Weyl's inequality, which gives $\max_{k=1,\dots,d_e}|\nu_k - \hat \nu_k | \le \|B-\hat B\|_2$, we obtain
\begin{align}
\sum_{k=1}^{d_e} ( \sqrt{\nu_k^+}  - \sqrt{\hat \nu_k^+})^2
\le 
\begin{cases}
d_e \nu_{d_e}^{-1} \|B-\hat B\|_2^2, & \text{if } \nu_{d_e} > 0; \\
d_e \|B-\hat B\|_2, & \text{if } \nu_{d_e} = 0.
\end{cases}
\end{align}

Combining the two bounds we derived, we conclude that
\begin{align}
\sum_{i=1}^n\|p_i - \hat p_i\|_2^2
\le C_2 
\begin{cases}
\|B-\hat B\|_2^2, & \text{if } \nu_{d_e} > 0; \\
\|B-\hat B\|_2, & \text{if } \nu_{d_e} = 0.
\end{cases}
\end{align}
where $C_2:=2\sqrt{\nu_1}C_1 + 2d_e(\nu_{d_e}^{-1} \vee 1)$.
Furthermore, we know that $B = HAH$ and that $\hat B = H\hat AH$, where $H=I-J/n$ is the $n\times n$ centering matrix, with $I$ being the identity matrix and $J$ being the matrix of ones. It follows that 
\begin{align}
\|B - \hat B\|_2^2 \le \|A - \hat A\|_2^2 \|H\|_2^4 \le n^2 \|A - \hat A\|_2^2,
\end{align} 
because $H$ has one zero eigenvalue, and $n-1$ eigenvalues equal to one.
And plugging this bound into the previous display, we obtain \eqref{hatu consistency}. 

\paragraph{Proof of \prpref{CS n-consistency}}
Let $\fB$ be the integral operator with kernel $b$, and $\lambda_1 \ge \lambda_2 \ge \cdots$ be the eigenvalues of $\fB$ with associated orthonormal eigenfunctions $\phi_1,\phi_2,\dots$. For any Borel measurable function $\pi: \bbQ \to \bbR^{d_e}$, the integral operator (denoted by $\fK$) with kernel $K: (q,q') \mapsto \<\pi(q), \pi(q')\>$ is positive semi-definite and has rank at most $d_e$. Let $\lambda_1' \ge \lambda_2' \ge \cdots$ be the eigenvalues of $\fK$ with associated eigenfunctions $\phi_1',\phi_2',\dots$. Note that $\lambda_{d_e}' \ge 0 = \lambda_{d_e+1}'$. Let $d_+$ be the number of eigenvalues of $\fB$ that are positive. 
We have 
\begin{align}
&\E\Big[\big(\<\pi(Q), \pi(Q')\> - b(Q,Q')\big)^2\Big] \\
& = \trace ((\fB - \fK)^2) \\
&= \trace (\fB^2) + \trace ( \fK^2) - 2 \trace (\fB\fK) \\
&\ge \sum_{k=1}^\infty  \lambda_k^2 + \sum_{k=1}^{d_e}  (\lambda_k')^2 - 2 \sum_{k=1}^{d_e \wedge d_+}  \lambda_k\lambda_k' =: \Gamma(\lambda, \lambda'),
%
\end{align}
where the last inequality is a consequence of \cite[Lem 4.25]{lim2022classical}. 
If $d_+ \ge d_e $, 
\begin{align}
\Gamma(\lambda, \lambda') & = \sum_{k=d_e+1}^\infty  \lambda_k^2 + \sum_{k=1}^{d_e}  (\lambda_k - \lambda_k')^2 ;
\end{align}
If $d_+ < d_e $, 
\begin{align}
\Gamma(\lambda, \lambda') & = \sum_{k=d_++1}^\infty  \lambda_k^2 + \sum_{k=d_++1}^{d_e}  (\lambda_k')^2 + \sum_{k=1}^{d_+}  (\lambda_k - \lambda_k')^2. 
\end{align}
In either case, the lower bound $\Gamma(\lambda, \lambda')$ is achieved when $\fK$ has kernel
\begin{equation}
K(q,q') = \sum_{k=1}^{d_e \wedge d_+} \lambda_k \phi_k(q) \phi_k(q'),
\end{equation}
or equivalently, $\pi: q\mapsto (\sqrt{\lambda_1^+} \phi_1(q),\dots,\sqrt{\lambda_{d_e}^+} \phi_{d_e}(q))$. 

%

Recall that $\mu$ is the probability measure of $Q$. Let $\mu_n$ be the empirical measure based on $Q_1,\dots,Q_n$. Let $A_n = (a_{ij})$ with $a_{ij} := -\frac12 \delta(Q_i,Q_j)^2$ and $B_n=H A_n H$. The domain of the CS embedding $\gamma_n: Q_i\mapsto (\sqrt{\nu_1^+}u_{i,1}, \dots, \sqrt{\nu_{d_e}^+} u_{i,d_e})$ can be extended to $\bbQ$ in the following way, as given in \cite{kroshnin2022infinite}. 
Since $\mu$ does not have an atom due to \eqref{measure zero}, with probability one, $Q_1,\dots,Q_n$ are distinct. Again by using \eqref{measure zero}, it is straightforward to verify that the admissibility condition as defined in \cite[Def 1]{geiss2013optimally} is satisfied for $\delta^4$. 
By \cite[Th 1, Th 2]{geiss2013optimally}, there exists an optimal transport map $T_n$ being the minimizer of
\begin{align}
W_4(\mu_n,\mu):=\inf_{T\in \cT (\mu,\mu_n)} \Big[\int_\bbQ \delta^4(q, T(q)) \mu(\d q)\Big]^{1/4},
\end{align}
where $\cT( \mu,\mu_n)$ is the set of transformations $T : \bbQ\to\{Q_1,\dots,Q_n\}$ such that $T(X) \sim \mu_n$ when $X \sim \mu$. Note that although the results in \cite{geiss2013optimally} are stated for a Euclidean space, their proofs are also valid for a compact metric space satisfying the admissible condition, as in our case. 
Denote $\cV_i=T_n^{-1}(Q_i)$ and note that $\mu(\cV_i)=1/n$, $i=1,\dots,n$. Now $\bbQ$ is partitioned into sets $\cV_1,\dots,\cV_n$ which are disjoint.
For any $q\in\bbQ$, define
\begin{align}
\pi_n: q\mapsto 
%
%
&\Big(\sqrt{\hat\lambda_1^+}\hat\phi_1(q), \dots, \sqrt{\hat\lambda_{d_e}^+} \hat\phi_{d_e}(q) \Big),
\end{align}
where $\hat\lambda_k = \nu_k/n$ and $\hat\phi_k =\sqrt{n}\sum_{i=1}^n u_{i,k} \I_{\cV_i}$. 
Note that $\pi_n$ coincides with the CS imbedding $\gamma_n$ when applied to $Q_1,\dots,Q_n$. 
We have 
\begin{align}
& \|\pi_n(Q) - \pi(Q)\|^2 \\
 & = \sum_{k=1}^{d_e} \big(\sqrt{\hat\lambda_k^+}\hat\phi_k(Q) -  \sqrt{\lambda_k^+}\phi_k(Q)\big)^2 \\
& \le 2 \sum_{k=1}^{d_e} (\sqrt{\lambda_k^+}\phi_k(Q) - \sqrt{\lambda_k^+}\hat \phi_k(Q) )^2 + 2 \sum_{k=1}^{d_e} ( \sqrt{\lambda_k^+}\hat \phi_k(Q) - \sqrt{\hat\lambda_k^+}\hat\phi_k(Q) )^2\\
& \le 2 \sqrt{\lambda_1} \sum_{k=1}^{d_e} (\phi_k(Q) - \hat \phi_k(Q))^2  + 2 \sum_{k=1}^{d_e} ( \sqrt{\lambda_k^+}  - \sqrt{\hat\lambda_k^+})^2 \hat\phi_k(Q)^2.
\end{align}
Notice that $\E_Q [\hat\phi_k(Q)^2] = n\sum_{i=1}^n u_{i,k}^2 \;\mu(\cV_i) =1$. 
It can be shown that $(\hat\lambda_k, \hat\phi_k)$, $k=1,\dots,d_e$ are eigenpairs of an operator $\fB_n$ defined as follows. Consider the operator $\fS_n: L^2(\bbQ,\mu) \to \bbR^n$ given by
\begin{align}
\fS_n(f) = \sqrt{n} \Big(\int_{\cV_1} f \d \mu, \dots, \int_{\cV_n} f \d \mu \Big).
\end{align}
Its adjoint operator $\fS_n^*: \bbR^n \to L^2(\bbQ,\mu)$ is given by 
\begin{align}
\fS_n^*:  y=(y_1,\dots,y_n)^\top \mapsto \sqrt{n}\sum_{i=1}^n y_i \I_{\cV_i} .
\end{align}
Let $\fB_n:=n^{-1}\fS_n^*B_n\fS_n$. It can be shown that $\fB_n$ has eigenvalues $\hat\lambda_1,\dots,\hat\lambda_{d_e}$ associated with orthonormal eigenfunctions $\hat\phi_1,\dots,\hat\phi_{d_e}$, all the remaining eigenvalues being zero. Following the same arguments as in the proof of \prpref{CS consistency}, we have for some constant $C>0$ only depending on the eigenvalues of $\fB$ and $d_e$,
\begin{align}
\E_Q \big[\min_{\pi\in\Pi} \|\pi_n(Q) - \pi(Q)\|^2\big] \le C
\begin{cases}
\|\fB_n - \fB\|_2^2, & \text{if } \lambda_{d_e} > 0; \\
\|\fB_n - \fB\|_2, & \text{if } \lambda_{d_e} = 0.
\end{cases}
\end{align}
It follows from \cite[Lem 5.7]{kroshnin2022infinite} that
\begin{align}
\|\fB_n - \fB\|_2 \le 2 \big(\E_{Q,Q'} \big[\delta^4(Q,Q')\big]\big)^{1/4} W_4 (\mu_n, \mu) + 2W_4^2 (\mu_n, \mu).
\end{align}
Hence the almost sure convergence to zero of $\E_Q \big[\min_{\pi\in\Pi} \|\pi_n(Q) - \pi(Q)\|^2\big]$ follows from that of $\E [W_4 (\mu_n, \mu)]$, and this holds when $\bbQ$ is compact and separable and $\mu$ is a Borel probability measure --- see \cite{weed2019sharp}. 


\paragraph{Proof of \prpref{Isomap n-consistency}}
Let $q_i = q_{i_0},q_{i_1},\dots,q_{i_m}=q_j$ be the shortest path connecting $q_i$ and $q_j$ in the neighborhood graph with connectivity radius $r_n$, so that $d_{ij} = \sum_{k=0}^{m-1} \delta(q_{i_k}, q_{i_{k+1}})$. 

For any $\eps>0$, there exists a path $\gamma_{ij}$ connecting $q_i$ and $q_j$ such that $a - \eps \le \delta_\L(q_i, q_j) \le a$, where $a = \L(\gamma_{ij})$. Since every path of finite length can be parameterized with unit speed \cite[Prop 2.5.9]{burago2001course}, we assume $\gamma_{ij}: [0,a] \to \bbQ$ is so. Let $\tilde q_k = \gamma_{ij}(ka/m)$ for $k=0,\dots,m$, where $m=\lceil 2a/r_n \rceil$. Let $q_{j_k}$ be the nearest neighbor of $\tilde q_k$ among the sample points so that $\delta(q_{j_k},\tilde q_k) \le \eps_n$, $k=0,\dots,m$. Note that $q_{j_k}$ and $q_{j_{k+1}}$ are connected, because 
\begin{align}
\delta(q_{j_k}, q_{j_{k+1}}) 
&\le \delta(q_{j_k}, \tilde q_{k}) + \delta(\tilde q_{k}, \tilde q_{k+1}) + \delta(\tilde q_{k+1}, q_{j_{k+1}}) \\
&\le \delta_\L(\tilde q_{k}, \tilde q_{k+1}) + 2\eps_n = a/m + 2\eps_n,
\end{align}
which is less than $r_n$ when $n$ is large enough. Using this, we have

\begin{align}
d_{ij} \le \sum_{k=0}^{m-1} \delta(q_{j_k}, q_{j_{k+1}}) 
%
%
&\le  \sum_{k=0}^{m-1}\delta_\L(\tilde q_{k}, \tilde q_{k+1}) + 2\eps_n m\\
& = a + 2\eps_n m \label{dij bound}\\
& \le  \delta_\L(q_i, q_j) +\eps + 2\eps_n m
\to \delta_\L(q_i, q_j) +\eps, \label{dij bound limit}
\end{align}
because $2\eps_n m \asymp \eps_n/r_n \to 0$ as $n\to\infty$.
Since $\eps>0$ is arbitrary, we conclude that $\limsup_n d_{ij} \le \delta_\L(q_i, q_j)$.

Suppose that $n$ is large enough that $\eps_n/r_n \le 1$ and $\eps_n\le a$. We derive from  \eqref{dij bound} that $d_{ij}\le a + 2(2a\eps_n/r_n + \eps_n) \le 7a$, which leads to $q_{i_k}\in \ball:=\{q\in\bbQ: \delta(q,q_i)\le 7a\}$ by the triangle inequality. Note that the closed ball $\ball$ is compact, since we assume $(\bbQ,\delta)$ is proper. It then follows from \eqref{delta approx} that $\omega$ is uniformly continuous on $\ball \times \ball$, and that $\delta_\L(q_{i_k}, q_{i_{k+1}}) = \delta(q_{i_k}, q_{i_{k+1}}) [1+\omega(q_{i_k}, q_{i_{k+1}})]$, where
\begin{align}
\eta_n := \sup_{k\in\{0,\dots,m-1\}} |\omega(q_{i_k}, q_{i_{k+1}})| \to 0.
\end{align}
since $\max_k \delta(q_{i_k}, q_{i_{k+1}}) \le r_n \to 0$.
This then yields that 
\begin{align}
\label{deltaLbound}
\delta_\L(q_i, q_j) 
&\le \sum_{k=0}^{m-1} \delta_\L(q_{i_k}, q_{i_{k+1}}) \\
&= \sum_{k=0}^{m-1} \delta(q_{i_k}, q_{i_{k+1}}) [1+\omega(q_{i_k}, q_{i_{k+1}})] 
\le d_{ij} (1+\eta_n),
\end{align}
from which we conclude that $\liminf_n d_{ij} \ge \delta_\L(q_i, q_j)$.

\paragraph{Proof of \prpref{Isomap sample n-consistency}}
Let $q_i = q_{i_0},q_{i_1},\dots,q_{i_m}=q_j$ be the shortest path connecting $q_i$ and $q_j$ in the neighborhood graph based on $(\delta_{ij})$ with connectivity radius $r_n/2$, with length $d_{ij}^\dagger := \sum_{k=0}^{m-1} \delta_{i_k, i_{k+1}}$. The assumption in \eqref{delta uniformly consistent} gives that, for any $\epsilon_1\in(0,1)$,
\begin{align}
\hat \delta_{i_k, i_{k+1}} = \delta_{i_k, i_{k+1}} \pm \epsilon_1 \delta_{i_k, i_{k+1}}, \text{ for all } k=1,\dots,m,
\end{align}
with probability converging to one. Conditional on the above event, we have $\hat \delta_{i_k, i_{k+1}} \le r_n$, which then implies that
\begin{align}
\hat d_{ij} \le \sum_{k=0}^{m-1} \hat \delta_{i_k, i_{k+1}} \le (1+\epsilon_1) \sum_{k=0}^{m-1} \delta_{i_k, i_{k+1}} = (1+\epsilon_1) d_{ij}^\dagger.
\end{align}
Let $d_{ij}^\ddagger$ be the shortest-path distance between $q_i$ and $q_j$ in the neighborhood graph based on $(\delta_{ij})$ with connectivity radius $2r_n$. Following the same argument as above, we have that, for any $\epsilon_2\in(0,1/2)$, 
\begin{align}
d_{ij}^\ddagger \le \frac{1}{1-\epsilon_2} \hat d_{ij},
\end{align}
with probability converging to one. According to \prpref{Isomap n-consistency}, both $d_{ij}^\dagger$ and $d_{ij}^\ddagger$ are consistent for $\delta_\L(q_i, q_j)$. Hence we must have that $\hat d_{ij}$ is consistent for $\delta_\L(q_i, q_j)$ as $n \to \infty$.

\paragraph{Proof of \prpref{intrinsic limit}}
For a metric $\delta$, let $\L_\delta$ and $\Sigma_\delta$ denote the quantities defined by $\delta$ in \eqref{length}. 

We first consider the first part of the statement.
It is enough to show that (1) $\delta_\L$ is finite; and (2) $\L_\delta(\gamma) = \L_A(\gamma)$ for every path $\gamma$ such that $\L_\delta(\gamma) < \infty$. 
Unless otherwise specified, the topology of reference is the ambient Euclidean metric. The corresponding norm will be denoted $\|\cdot\|$, as usual.
We will use the fact that, by our assumption that $\delta$ is equivalent to the Euclidean metric, $(q,q_0) \mapsto \delta(q,q_0)$ is continuous. 

For (1), for any two (distinct) points, $q_0, q_1 \in \bbQ$, consider a smooth path $\gamma: [0,1] \to \bbQ$ connecting them. The existence of this path is elementary and rests the fact that $\bbQ$ is open and connected. 
By the fact that $\delta$ is continuous and property \eqref{intrinsic limit}, $u(s,t) := \delta(\gamma(s),\gamma(t))/\|A(\gamma(t))^{1/2}(\gamma(s)-\gamma(t))\|$ is continuous on $[0,1]$, and therefore uniformly continuous. 
In particular, since $u(t) := u(t,t) = 1$ for all $t$, there is $\eta > 0$ that $u(s,t) \le 2$ when $|t-s| \le \eta$.
Let $M_1 := \max_t \|A(\gamma(t)\| < \infty$ and $M_2 := \max_t \|\gamma'(t)\|$.
For a grid $0 = t_0  < t_1 < \dots < t_k = 1$ with $\max_i (t_i -t_{i-1}) \le \eta$, we have 
\begin{align}
\Sigma_\delta(\gamma, t_1, \dots, t_k)
&= \sum_{i=1}^k \delta(\gamma(t_{i-1}), \gamma(t_i)) \\
&\le 2 \sum_{i=1}^k \|A(\gamma(t_i))^{1/2} (\gamma(t_{i-1}) -\gamma(t_i))\| \\
&\le 2 M_1 \sum_{i=1}^k \|\gamma(t_{i-1}) -\gamma(t_i)\| \\
&\le 2 M_1 M_2 \sum_{i=1}^k (t_i-t_{i-1}) 
= 2 M_1 M_2. 
\end{align}
Taking the supremum over such grids, we deduce that $\L_\delta(\gamma) \le 2 M_1 M_2 < \infty$.


For (2), consider an arbitrary path $\gamma : [0,1] \to \bbQ$ such that $\L_\delta(\gamma) < \infty$. Since $\gamma$ is $\delta$-continuous, it is also $\|\cdot\|$-continuous, and furthermore, uniformly so, since $[0,1]$ is compact. Therefore, $u$ defined above is uniformly continuous.
Thus, coupled with the fact that $u$ is strictly positive, for $\eps>0$, there is $\eta > 0$ such that, if $|t-s|\le \eta$ then $1 - \eps \le u(s,t) \le 1+\eps$. 
For a grid $0 = t_0 < t_1 < \dots < t_k = 1$ with $\max_i (t_i -t_{i-1}) \le \eta$, we have
\begin{align}
\Sigma_\delta(\gamma, t_1, \dots, t_k)
&= \sum_{i=1}^k \delta(\gamma(t_{i-1}), \gamma(t_i)) \\
&= (1\pm\eps) \sum_{i=1}^k \|A(\gamma(t_i))^{1/2}(\gamma(t_{i-1}) -\gamma(t_i))\| \\
&= (1\pm \eps) \Sigma_A(\gamma, t_1, \dots, t_k). 
\end{align}
Taking the supremum over all such grids, we obtain
\begin{align}
\L_\delta(\gamma) = (1\pm \eps) \L_A(\gamma),
\end{align}
and $\eps > 0$ being arbitrary, we conclude that $\L_\delta(\gamma) = \L_A(\gamma)$.

We now consider the second part of the statement. 
Since we have already established the first part, it suffices to show that $\delta_\L$ also satisfies \eqref{intrinsic limit} based on the fact that it corresponds to the Riemannian metric with tensor $A$. When $A$ is twice differentiable, then such a result follows immediately from the fact that shortest paths are geodesics and a geodesic has curvature bounded by the maximum sectional curvature of the manifold along its travel path. (See \cite[Sec~3]{arias2019unconstrained} in the context of an embedded manifold, although this may be considered general due to Nash's theorem.) 
When $A$ is only continuous, we rely on \lemref{bruce} below.

\begin{lemma}[Bruce Driver, personal communication]
\label{lem:bruce}
Suppose an open connected set $\bbQ \subset \bbR^d$ is equipped with a continuous Riemannian metric tensor $A$. Then the resulting metric on $\bbQ$, denoted $\delta_A$, satisfies 
\begin{align}
\delta_A(q,q_0)
= \|A^{1/2}(q_0) (q-q_0)\| [1 \pm \omega(q,q_0)],
\end{align}
where 
\begin{align}
\omega(q,q_0) 
= \max\big\{\|A^{1/2}(q') A^{-1/2}(q_0) - I\| : q' \in \bar\ball(q_0, \|q-q_0\|)\big\},
\end{align}
for all $q_0, q \in \bbQ$ such that $\cE(q, q_0) \subset \bbQ$,
\begin{align}
\cE(q, q_0) := \big\{q': \|A^{1/2}(q_0)(q'-q_0)\| \le \|A^{1/2}(q_0)(q-q_0)\|\big\}.
\end{align}
Note that $\omega$ is a continuous function satisfying $\omega(q,q) = 0$ for all $q \in \bbQ$.
\end{lemma}

\begin{proof}
Fix such a pair of points $q_0, q \in \bbQ$, and define $v := q-q_0$ and $A_0 := A(q_0)$.

For the upper bound, consider the line segment $\sigma_t = q_0+tv$ for $t \in [0,1]$. Since this line segment is within $\bbQ$, we have $\delta_A(q,q_0) \le \L_A(\sigma)$, with
\begin{align}
\L_A(\sigma)
&= \int_{0}^{1} \|A^{1/2}(\sigma_t) \dot\sigma_t\| \d t \\
&= \int_{0}^{1} \|A_0^{1/2} v + (A^{1/2}(\sigma_t)-A_0^{1/2}) v\| \d t \\
&\le \|A_0^{1/2} v\| + \max_{t \in [0,1]} \|(A^{1/2}(\sigma_t)-A_0^{1/2}) A_0^{-1/2}\| \|A_0^{1/2} v\| \\
&\le \|A_0^{1/2} v\| [1 + \omega(q,q_0)], 
\end{align}
by the triangle inequality and then the fact that $\sigma([0,1]) \subset \cE(q,q_0)$.

For the lower bound, consider a shortest path $\sigma:[0,1] \to \bbQ$ joining $\sigma_0 = q_0$ and $\sigma_1 = q$. Let 
\begin{align}
\tau = \min\{t : \|A_0^{1/2}(\sigma_t-q_0)\| = \|A_0^{1/2} v\|\big\}.
\end{align}
Then $\delta_A(q,q_0) = \L_A(\sigma)$, with
\begin{align}
\L_A(\sigma)
&= \int_{0}^{1} \|A^{1/2}(\sigma_t) \dot\sigma_t\| \d t \\
&\ge \int_{0}^{\tau} \|A_0^{1/2} \dot\sigma_t + (A^{1/2}(\sigma_t)-A_0^{1/2}) \dot\sigma_t\| \d t \\
&\ge \int_{0}^{\tau} \|A_0^{1/2} \dot\sigma_t\| \big[1 - \|(A^{1/2}(\sigma_t)-A_0^{1/2}) A_0^{-1/2}\| \big] \d t \\
&\ge \int_{0}^{\tau} \|A_0^{1/2} \dot\sigma_t\| \d t \ \big[1 - \max_{s \in [0,\tau]} \|(A^{1/2}(\sigma_s)-A_0^{1/2}) A_0^{-1/2}\| \big] \\
&\ge \|A_0^{1/2} v\|\ [1 - \omega(q,q_0)], 
\end{align}
by the triangle inequality and then the fact that $\sigma([0,\tau]) \subset \cE(q,q_0)$ together with
\begin{align}
\|A_0^{1/2} v\|
= \|A_0^{1/2} (\sigma_\tau - \sigma_0)\| 
= \Big\|\int_0^\tau A_0^{1/2} \dot\sigma_t \d t\Big\|
\le \int_{0}^{\tau} \|A_0^{1/2} \dot\sigma_t\| \d t,
\end{align}
the last inequality being Jensen's.
\end{proof}

\paragraph{Verifying the claims made in \remref{intrinsic limit}}
For the first part, since $\delta$ is assumed continuous with respect to the Euclidean topology, it suffices to show that, if $\theta_0$ and $(\theta_n)$ are inside $\Theta$ and such that $\delta(\theta_n, \theta_0) \to 0$ as $n \to \infty$, then $\|\theta_n - \theta_0\| \to 0$ as well. Since $\Theta$ is bounded, extracting a subsequence if needed, we may assume that $(\theta_n)$ converges to some $\theta_\infty \in \bar\Theta$ in $\|\cdot\|$. Since $\delta$ is continuous, we have $\delta(\theta_\infty, \theta_0) = 0$, and because $\delta$ is assumed to be a metric not only on $\Theta$, but on its closure as well, this implies that $\theta_\infty = \theta_0$.     

For the second part, a Taylor expansion of order~2 gives
\begin{align}
\delta(q_1+\eps_1, q_0+\eps_0)^2 - \delta(q_1,q_0)^2
&= c_1(q_1,q_0)^\top \eps_1 + c_0(q_1,q_0)^\top \eps_0 \\
&\quad + \eps_1^\top C_{11}(q_1,q_0) \eps_1 + \eps_0^\top C_{00}(q_1,q_0) \eps_0 + \eps_1^\top C_{10}(q_1,q_0) \eps_0 \\
&\qquad + R(q_1,q_0) (\|\eps_1\|^2 + \|\eps_0\|^2),
\end{align}
where all the functions just introduced are continuous and $R(q,q) = 0$ for all $q$. In particular, taking $\eps_0=0$, $q_1=q_0$, $\eps_1 = q-q_0$, we get
\begin{align}
\delta(q, q_0)^2
= c_1(q_0,q_0)^\top (q-q_0) 
+ (q-q_0)^\top C_{11}(q_0,q_0) (q-q_0) 
+ R(q_0,q_0) \|q-q_0\|^2.
\end{align}
Because $\delta^2$ is nonnegative, it must be the case that $c_1(q_0,q_0) = 0$, and because $\delta^2$ has nonsingular Hessian, $C_{11}(q_0,q_0)$ must be nonsingular and thus positive definite.
This establishes \eqref{intrinsic limit} with $A(q_0) = C_{11}(q_0,q_0)$, half the Hessian of $\delta$ at $(q_0,q_0)$.

\paragraph{Proof of \prpref{time warping}}
Let $\bar\Theta$ be the space of all non-decreasing functions from $[0,1]$ to $[0,1]$. Here $\Theta\subset\bar\Theta$. We note that the definition $\Delta_{\rm W}(\theta, \theta_0)$ in \eqref{W2 1D} may extended to $\bar\Theta$. By Helly's selection theorem, for any sequence $\theta_1,\theta_2,\dots$ in $\bar\Theta$, there exists a subsequence $\theta_{n_1},\theta_{n_2},\dots$ and a function $\theta$ on $[0,1]$ such that $\lim_{k\to\infty} \theta_{n_k}(x) \to \theta(x)$ for all $x\in[0,1]$, which implies $\lim_{k\to\infty}\Delta_{\rm W}( \theta_{n_k},\theta) \to 0$ by the dominated convergence theorem. Note that $0\le \theta_{n_k}(x) \le 1$ for all $x\in[0,1]$ and $\theta_{n_k}(x_1) \le \theta_{n_k}(x_2)$ for $0\le x_1\le x_2 \le 1$. By taking the limit as $k\to\infty$, we have $0\le \theta(x) \le 1$ for all $x\in[0,1]$ and $\theta(x_1) \le \theta(x_2)$ for $0\le x_1\le x_2 \le 1$. In other words, $\theta\in\bar\Theta$. Therefore, $\bar\Theta$ is compact for the $\Delta_{\rm W}$ metric.
It turns out that $\bar\Theta$ is the closure of $\Theta$ --- see \lemref{barTheta}.

Let $\mu$ be a Borel probability measure on the compact metric space $(\bar\Theta,\Delta_W)$ such that $\mu(\bar\Theta\setminus\Theta)=0$. 
Note that this is only possible because $\bar\Theta$ is the closure of $\Theta$. Let $\E$ be the expectation under $\mu$. Similar to \eqref{bqdef}, with independent random elements $\vartheta,\vartheta'\sim\mu$, define
\begin{align}
\beta(\theta,\theta') := -\frac12 \Big(\Delta_{\rm W}(\theta,\theta')^2 - \E[\Delta_{\rm W}(\theta,\vartheta')^2] - \E[\Delta_{\rm W}(\vartheta,\theta')^2] + \E[\Delta_{\rm W}(\vartheta,\vartheta')^2]\Big),
\end{align}
and let $\fB:L^2(\mu)\to L^2(\mu)$ be the operator given by $\fB(\psi)(\theta) = \int_\Theta \beta(\theta,\theta') \psi(\theta') \d \mu(\theta')$. By \cite[Lem 4.2]{lim2022classical}, $\beta$ is a self-adjoint and compact (in fact, Hilbert-Schmidt) operator. It is also a continuous kernel, because it can be seen that as $\Delta_{\rm W}(\theta_n,\theta) \to 0$ and $\Delta_{\rm W}(\theta_n',\theta') \to 0$, 
\begin{align*}
|\Delta_{\rm W}(\theta_n,\theta'_n) - \Delta_{\rm W}(\theta,\theta')| & \le |\Delta_{\rm W}(\theta_n,\theta'_n) - \Delta_{\rm W}(\theta,\theta'_n)| + |\Delta_{\rm W}(\theta,\theta'_n) - \Delta_{\rm W}(\theta,\theta')| \\
& \le \Delta_{\rm W}(\theta_n,\theta) + \Delta_{\rm W}(\theta',\theta'_n) \to 0,
\end{align*}
and
\begin{align*}
|\E[\Delta_{\rm W}(\theta,\vartheta')^2] - \E[\Delta_{\rm W}(\theta_n,\vartheta')^2]| &= |\E[\Delta_{\rm W}(\theta,\vartheta')^2 - \Delta_{\rm W}(\theta_n,\vartheta')^2]| \\
& \le \E[|\Delta_{\rm W}(\theta,\vartheta') +  \Delta_{\rm W}(\theta_n,\vartheta')| \cdot |\Delta_{\rm W}(\theta,\vartheta') -  \Delta_{\rm W}(\theta_n,\vartheta')|]\\
& \le 2 \Delta_{\rm W}(\theta,\theta_n) \to 0,
\end{align*}
where we have used the fact that $\Delta_{\rm W}(\theta,\theta') \le 1$ for any $\theta,\theta'\in\bar\Theta$ and the triangle inequality.

Suppose that $\lambda\neq 0$ is an eigenvalue of $\fB$, and $\phi \in L^2(\mu)$ be a normalized eigenfuncton for $\lambda$. It is known from \cite[Lem 4.2]{lim2022classical} that $\int_{\bar\Theta} \phi(\theta) \d \mu(\theta) = 0$. Using this and the definition of eigenvalues, we have 
\begin{align}
\lambda & = \int_{\bar\Theta} \int_{\bar\Theta} \beta(\theta,\theta') \phi(\theta') \phi(\theta) \d\mu(\theta') \d\mu(\theta) \\
& = -\frac12 \int_{\bar\Theta} \int_{\bar\Theta} \Delta(\theta,\theta')^2 \phi(\theta') \phi(\theta)\d\mu(\theta') \d\mu(\theta) \\
& = -\frac12 \int_{\bar\Theta} \int_{\bar\Theta} \int_0^1 (\theta(x)-\theta'(x))^2 f(x) \d x \phi(\theta') \phi(\theta)\d\mu(\theta') \d\mu(\theta) \\
& =  \int_{\bar\Theta} \int_{\bar\Theta} \int_0^1 \theta(x)\theta'(x) f(x) \d x \phi(\theta') \phi(\theta) \d\mu(\theta') \d\mu(\theta) \\
& =  \int_0^1 \Big[\int_{\bar\Theta} \theta(x) \phi(\theta) \d x\d\mu(\theta) \Big]^2 f(x) \d x \\
& \ge 0.
\end{align}
This implies that all the eigenvalues of $\fB$ are non-negative. Let $\lambda_1 \ge \lambda_2 \ge \cdots$ be the eigenvalues of $\fB$ with associated orthonormal eigenfunctions $\phi_1,\phi_2,\dots$. The $d_e$-dimensional CS embedding is given by 
\begin{align}
\pi: \theta \mapsto (\sqrt{\lambda_1}\phi_1(\theta),\dots, \sqrt{\lambda_{d_e}}\phi_{d_e}(\theta)).
\end{align}
On the one hand, by definition,
\begin{align}
\Delta^{\rm CS}_{d_e}(\theta,\theta_0)^2 = \|\pi(\theta) - \pi(\theta_0)\|^2 = \sum_{k=1}^{d_e} \lambda_k \phi_k(\theta) \phi_k(\theta) + \sum_{k=1}^{d_e} \lambda_k \phi_k(\theta_0) \phi_k(\theta_0) - 2\sum_{k=1}^{d_e} \lambda_k \phi_k(\theta) \phi_k(\theta_0).
\end{align}
On the other hand, for any $\theta,\theta_0\in\Theta$,
\begin{align}
\Delta_{\rm W}^2(\theta,\theta_0) 
& = \beta(\theta,\theta) +  \beta(\theta_0,\theta_0) - 2  \beta(\theta,\theta_0) \\
&= \sum_{k=1}^\infty \lambda_k \phi_k(\theta) \phi_k(\theta) + \sum_{k=1}^\infty \lambda_k \phi_k(\theta_0) \phi_k(\theta_0) - 2\sum_{k=1}^\infty \lambda_k \phi_k(\theta) \phi_k(\theta_0).
\end{align}
The difference of the above two expressions gives
\begin{align}
\Delta_{\rm W}^2(\theta, \theta_0) - \Delta^{\rm CS}_{d_e}(\theta,\theta_0)^2 & = \sum_{k=d_e+1}^\infty \lambda_k \phi_k(\theta) \phi_k(\theta) + \sum_{k=d_e+1}^\infty \lambda_k \phi_k(\theta_0) \phi_k(\theta_0) - 2\sum_{k=d_e+1}^\infty \lambda_k \phi_k(\theta) \phi_k(\theta_0) \\
& = \sum_{k=d_e+1}^\infty \lambda_k [\phi_k(\theta) - \phi_k(\theta_0)]^2 \ge 0.
\end{align}
Clearly, for any non-negative numbers $a\ge b$, we have $a-b = \sqrt{a^2 - 2ab + b^2} \le \sqrt{a^2 - b^2}$. Hence,
\begin{align}
\sup_{\theta, \theta_0 \in \Theta} \big|\Delta^{\rm CS}_{d_e}(\theta,\theta_0) - \Delta_{\rm W}(\theta, \theta_0)\big| 
&\le \sup_{\theta, \theta_0 \in \Theta} \sqrt{\sum_{k=d_e+1}^\infty \lambda_k [\phi_k(\theta) - \phi_k(\theta_0)]^2} \\
&\le 2 \sup_{\theta \in \Theta}  \sqrt{\sum_{k=d_e+1}^\infty \lambda_k \phi_k(\theta)^2}
\mathop{\longrightarrow}_{d_e \to \infty} 0,
\end{align}
where the limit at the end follows from \cite[Lem 4.6.6]{hsing2015theoretical}.

\begin{lemma}
\label{lem:barTheta}
$\bar\Theta$ is the closure of $\Theta$ for $\Delta_{\rm W}$.
\end{lemma}

\begin{proof}
Consider a twice differentiable density function supported on $[-1,1]$, for example, the triweight kernel $g(x) = \frac{35}{32}(1-x^2)^3 \I(|x|\le1)$. For any positive bandwidth $h$, let $g_h(x) = h^{-1} g(x/h)$. For any $\bar\theta\in\bar\Theta$, extend it to $\mathbb{R}$ by setting $\bar\theta(x) = 0$ for $x\in(-\infty,0)$ and $\bar\theta(x) = 1$ for $x\in(1,\infty)$, and define 
\begin{align}
\tilde \theta (x) = (\bar\theta*g_h) \big((1+2h)x-h\big) = \frac{1}{h}\int_{-h}^h \bar\theta\big((1+2h)x-h - y\big) g\Big(\frac{y}{h}\Big) \d y , \quad \text{for } x\in[0,1].
\end{align}
Note that $\tilde \theta (0)=0$, $\tilde \theta (1)=1$, $\tilde \theta$ is twice differentiable because $g$ is, and that $\tilde \theta$ is non-decreasing because $\bar\theta$ is. Furthermore, again based on the same fact that $\bar\theta$ is non-decreasing, for any $x\in[0,1]$, we have
\begin{align}
\tilde \theta (x) 
&\le \frac{1}{h}\int_{-h}^h \bar\theta\big((1+2h)x \big) g\Big(\frac{y}{h}\Big) \d y = \bar\theta\big((1+2h)x\big) \ge \bar\theta(x), \\
\tilde \theta (x) 
&\ge \frac{1}{h}\int_{-h}^h \bar\theta\big((1+2h)x - 2h \big) g\Big(\frac{y}{h}\Big) \d y = \bar\theta\big((1+2h)x - 2h\big) \le \bar\theta(x).
\end{align}
Hence, as $h\to 0$,
\begin{align}
\Delta_{\rm W}^2(\bar\theta,\tilde \theta) &\le \int_0^1 \Big[ \max\Big( \bar\theta\big((1+2h)x\big) - \bar\theta(x), \bar\theta(x) - \bar\theta\big((1+2h)x - 2h\big)\Big)\Big]^2 f(x) \d x \\
&\to \int_0^1 \Big[ \max\big( \bar\theta(x+) - \bar\theta(x), \bar\theta(x) - \bar\theta(x-)\big)\Big]^2 f(x) \d x \\
&=0, \label{deltabartilde}
\end{align}
where the convergence follows from the dominated convergence theorem, and the last equality is due to the fact that $\bar\theta$ --- as a non-decreasing function --- has at most countable discontinuous points, which form a set of zero Lebesgue measure.

Notice that $\tilde \theta$ may not belong to $\Theta$ because it is possible that $\tilde \theta$ is not strictly increasing. 
Below we use $\tilde\theta$ to construct a function $\theta\in\Theta$ such that $\Delta_{\rm W}(\theta,\bar\theta)$ is small. Let $S\subset[0,1]$ be the set of points with a neighborhood where $\bar\theta$ is not flat, that is, for any $x\in S$, there exists $\eta>0$ such that $\bar\theta(x_1)<\bar\theta(x_2)$ for all $x_1\in[x-\eta,x)$ and $x_2\in(x,x+\eta]$. For any $z\in\bbR$, denote the intervals 
\begin{align}
I(z) = [(1+2h)z-2h, (1+2h)z] \text{ and } J(z) = \Big[\frac{z}{1+2h}, \frac{z+2h}{1+2h}\Big].
\end{align} 
Notice that $z\in J(x) \Leftrightarrow x\in I(z)$. Fix $x\in S$ and suppose that $z_1,z_2\in J(x)$ with $z_1<z_2$, which implies that $x\in I(z_1)\cap I(z_2)$ and thus $(1+2h)z_2-2h < x < (1+2h)z_1$. We have 
\begin{align}
& \tilde\theta (z_1) - \tilde\theta (z_2) \\
&= \frac{1}{h}\int_{-h}^h \Big[\bar\theta\big((1+2h)z_1-h - y\big) - \bar\theta\big((1+2h)z_2-h - y\big)\Big] g\Big(\frac{y}{h}\Big) \d y \\
&= \int_{-1}^1 \psi(u) g(u) \d u, \label{thetatildediff}
\end{align}
where we use a change of variable $y=hu$ in the last step and 
\begin{align}
\label{psiu}
\psi(u) := \bar\theta\big((1+2h)z_1-h (1- u)\big) - \bar\theta\big((1+2h)z_2-h (1- u)\big).
\end{align}
Note that $\psi(u) g(u) \le 0$ for $u\in(-1,1)$ because $\bar\theta$ is non-decreasing and $g$ is positive on $(-1,1)$, and that $\psi(u) g(u) < 0$ for $u$ in some neighborhoods of $u_i:=1-[(1+2h)z_i + x]/h$, $i=1,2$ (which make one of the two terms on the right side of \eqref{psiu} equal to $\bar\theta(x)$, respectively), by the definition of $x$. Hence $\tilde\theta (z_1) - \tilde\theta (z_2)<0$, which means that $J(x)$ is an interval where $\tilde\theta$ is strictly increasing.

Now let $T\subset[0,1]$ be the set of points with a neighborhood where $\tilde\theta$ is not flat. For any $z\in T$, there exists $\eta>0$ such that $\tilde\theta(z_1)<\tilde\theta(z_2)$ for all $z_1 \in [z-\eta,z)$ and $z_2 \in (z,z +\eta]$. The calculation in \eqref{thetatildediff} implies that there exists $u_0\in(-1,1)$ depending on $z_1$ and $z_2$ such that $\psi(u_0)<0$. In other words, there exists $x\in S$ such that $x\in ((1+2h)z_1 - h(1-u_0), (1+2h)z_2 - h(1-u_0))$. Since $z_1$ and $z_2$ can be arbitrarily close to $z$, we have that there exists $x\in S$ such that $x\in I(z)$, or equivalently, $z\in J(x)$, which is an interval of length $(2h)/(1+2h)$ where $\tilde\theta$ is strictly increasing, as pointed out above.

Without loss of generality, suppose that there exists a sequence $0 =a_1< b_1 < a_2 < b_2 < \cdots < a_m < b_m =1$ such that $\tilde \theta(b_i) = \tilde\theta(a_{i+1})$, which means $\tilde\theta(x)$ is flat on $[b_i,a_{i+1}]$, and that $\tilde\theta$ is strictly increasing on $[a_i,b_i]$. Here $b_i-a_i \ge (2h)/(1+2h)$, and hence $m$ must be finite for any fixed $h>0$. 

For any $\epsilon$ small enough that $\epsilon < \min_i [\tilde\theta(b_i) - \tilde\theta(a_i)]$, let $b_i^* = \tilde\theta^{-1}(\tilde\theta(b_i) - \epsilon)$ for $i=1,\dots,m-1$, and $a_i^* = \tilde\theta^{-1}(\tilde\theta(a_i) + \epsilon)$ for $i=2,\dots,m$. Note that $\tilde\theta^{-1}(x)$ is well defined for $x\in \bigcup_{i=1}^m (a_i,b_i)$. Define 
\begin{align}
\check\theta(x) = 
\begin{cases}
\tilde\theta(x) & x\in \big(\bigcup_{i=1}^m[a_i^*,b_i^*] \big) \bigcup [a_1,b_1^*] \bigcup [a_m^*,b_m]\\
\phi(x) & x\in \bigcup_{i=2}^{m-1}[b_i^*, a_{i+1}^*] 
\end{cases},
\end{align}
where $\phi(x)$ is a line segment connecting the two points $(b_i^*,\tilde\theta(b_i^*))$ and $(a_{i+1}^*,\tilde\theta(a_{i+1}^*))$. Under such a construction, $\check\theta$ is strictly increasing, and 
\begin{align}
\label{deltatildecheck}
\Delta_{\rm W}^2(\check\theta,\tilde \theta) \le \int_{\bigcup_{i=2}^{m-1}[b_i^*, a_{i+1}^*]}  |\phi(x) - \tilde\theta(x)|^2 f(x) \d x \le \epsilon^2.
\end{align}

Next we construct $\theta$ from $\check\theta$ by using convolution with a kernel, in a same way as for $\tilde\theta$ from $\bar\theta$. Extend $\check\theta$ to $\mathbb{R}$ such that $\check\theta(x)\equiv0$ for $x\in(-\infty,0)$ and $\check\theta(x)\equiv1$ for $x\in(1,\infty)$, and for $\rho>0$ define 
\begin{align}
\theta (x) = (\check\theta*g_\rho) \big((1+2\rho)x-\rho\big), \;\; x\in[0,1].
\end{align}
Again, $\theta (0)=0$, $\theta (1)=1$, and $\tilde \theta$ is twice differentiable. Most importantly, $\theta$ is strictly increasing because $\check\theta$ is, following from the same arguments for the strictly increasing part of $\tilde\theta$. Hence $\theta\in\Theta$. Similar to \eqref{deltabartilde}, we have
\begin{align}
\Delta_{\rm W}^2(\theta,\check\theta) \to 0 \text{ as } \rho\to 0. 
\end{align}
Combing this with \eqref{deltabartilde} and \eqref{deltatildecheck}, we see that $\Delta_{\rm W}^2(\theta,\bar\theta)$ can be made arbitrarily small by (sequentially) choosing $h$, $\epsilon$ and $\rho$ small enough, which means that $\bar\Theta$ is the closure of $\Theta$. 
\end{proof}


\small
\bibliographystyle{chicago}
\bibliography{ref}



\end{document}